\documentclass[12pt]{amsart}
\usepackage{amsmath, amsthm, amssymb, amsfonts}
\usepackage{graphicx}
\newtheorem{problem}{Problem}
\newtheorem{theorem}{Theorem}
\newtheorem{lemma}{Lemma}
\newtheorem{corollary}{Corollary}
\newtheorem{conjecture}{Conjecture}
\theoremstyle{definition}
\newtheorem{remark}{Remark}
\newcommand{\RR}{\mathbb{R}}
\newcommand{\CC}{\mathbb{C}}
\newcommand{\DD}{\mathbb{D}}

\begin{document}

\title{Estimating the Koebe radius for polynomials}

\author{Dmitriy Dmitrishin}
\address{Odessa National Polytechnic University, 1 Shevchenko Ave.,
Odessa 65044, Ukraine}
\email{dmitrishin@opu.ua}

\author{Andrey Smorodin}
\address{Odessa National Polytechnic University, 1 Shevchenko Ave.,
Odessa 65044, Ukraine}

\author{Alex Stokolos}
\address{Georgia Southern University, Statesboro GA, 30460}
\email{astokolos@geogriasouthern.edu}

\begin{abstract}

For a pair of conjugate trigonometrical polynomials \(C (t) = \sum_ { j = 1 } ^N { { a_j}\cos jt } \), \(S(t) = \sum_ { j = 1 } ^N { { a_j}\sin jt } \) with real coefficients and  normalization \({a_1} = 1 \) we solve the  extremal problem
\[
 \sup_ {a_2,...,a_N} 
\left ({ \min_t 
\left\{ {\Re \left ({ F\left ({ { e^ {it} } } \right) } \right): \Im 
\left ({ F\left ({ { e^ {it} } } \right) } \right) = 0 } \right\} } \right) = -\frac14 \sec ^2\frac\pi{N + 2}.  
\]
We show that the solution is unique and is given by 
\[
a_j^ {(0)} = \frac {1} { { {  U'_N}\left ({\cos \frac{\pi } { { N + 2 } } } \right) } } { U' _ { N - j + 1 } }\left ({\cos \frac{\pi } { { N + 2 } } } \right) { U_ { j - 1 } }\left ({\cos \frac{\pi }
 { { N + 2 } } } \right),  
\]
where the \({ U_j}\left (x \right) \) are the Chebyshev polynomials of the second kind, $j = 1, \ldots, N.$ As a consequence, we obtain some theorems on covering of intervals by polynomial images of the unit disc. We formulate several conjectures on  a number of extremal problems on  classes of polynomials. 
\end {abstract}


\maketitle

\section {Introduction}

One of the fundamental problems in nonlinear dynamics is suppression of the chaotic regime of a system \cite{1.,2.}.  To solve this problem, various schemes connected with special representation of the delayed feedback (DFC) have been suggested \cite{3.,4.,5.}. The classical linear DFC  scheme \cite{6.}  is only applicable for a limited range of parameters involved in the nonlinear system \cite{7.,8.}. On the other hand, the nonlinear  DFC scheme with several delays allows solving the problem of controling chaos in a fairly general case  \cite{9.,10.}. At the same time, it is necessary to keep the length of the prehistory applied at a minimal possible level. The problem reduces to the following optimization problem:  find
\begin{equation}\label{1}
G_N =  \sup_{\sum_ { j = 1 } ^N {a_j } = 1 } \min_t \, \{  \Re  ( F ( e^ {it} )): \Im ( F ( e^ {it})) = 0  \} , 
\end{equation}
where the supremum is taken over all  real $a_1,\ldots,a_n$ summing to~1 and where \(F\left( z \right) = \sum_{j = 1}^N {{a_j}{z^j}}\),
and also find a polynomial $F$ that achieves the supremum. The coefficients \(a_j^0 \)  of this polynomial determine the parameters of the required DFC. 

This problem is solved in \cite{10.}:
\begin{align*}
G_N &= - {\tan ^2 }\frac{\pi } { { 2 (N + 1) } },\\ 
a_j^0 &= \frac {{2N} } { { N + 1 } }\tan \frac{\pi } { { 2 (N + 1) } }\left ({ 1 - \frac { { j - 1 } }{N} } \right) \sin \frac {{j\pi}} { { N + 1 } }, \quad j = 1, \ldots, N,
\end{align*}
and the coefficients \(a_j^0 \) can be expressed through Chebyshev polynomials of the second kind:
\begin{align*}
a_ {1} ^0 &= \frac {{2N} } { { N + 1 } }\left ({ 1 - \cos \frac{\pi } { { N + 1 } } } \right),\\
a_j^0 &= a_ {1} ^0\left ({ 1 - \frac { { j - 1 } }{N} } \right) { U_ { j - 1 } }\left ({\cos \frac{\pi } { { N + 1 } } } \right), \quad j = 2, \ldots, N, 
\end{align*}
where 
\[
U_j\left (x \right) = { U_j}\left ({\cos t } \right) = \frac { {\sin (j + 1) t } } { {\sin t } } = {2^j}{x^j} +  \ldots.
\]
The polynomial
\begin{equation}\label{2}
S\left (z \right) = a_{1}^0\sum_{j = 1}^N{\left ({1 - \frac{{j - 1}}{N}}\right){U_{j - 1}}\left ({\cos \frac{\pi}{{N + 1}}}\right){z^j}}
\end{equation}
 was first considered  in \cite{11.}  and is called the \emph{Suffridge polynomial.}  The polynomial is univalent. Various properties of these polynomials were  obtained in \cite{12.,13.}. 
 
The Suffridge polynomial has \(S\left (0 \right) = 0 \), \(S\left (1 \right) = 1 \), so it is natural to restrict problem (1) to polynomials with this normalization. Such an optimization problem also appears in linear feedback control. 

So, let us consider the following problem:

\begin{problem}\label{p1}
 Find
\[
J_N = \sup_{ a_2,...,a_N\in\RR}\left ({\min_t\left\{{\Re\left ({F\left ({{e^{it}}}\right)}\right):\Im\left ({F\left ({{e^{it}}}\right)}\right) = 0}\right\}}\right) 
\]
where \(F\left( z \right) = z+\sum_{j = 2}^N{{a_j}{z^j}}\), and find the extremal polynomial $F$.
\end{problem}

It is clear that  \(\Im \left({F\left({{e^{it}}}\right)}\right) = 0\) when \(t = 0\) or \(t = \pi \).

\section{Auxiliary results}

First, we prove a property of the image of the unit disc  under a polynomial mapping.  

\begin{lemma}
Let \(F(z) = z +{a_2}{z^2}+\ldots +{a_N}{z^N}\)  \((a_j \in \mathbb C,\, j = 2, \ldots, N) \), and let \(\mathbb D = \left\{{z:\left| z \right| < 1}\right\}\) be the unit disc. Then  \(F(\mathbb D) \) contains  a disc of radius \(\frac{1}{N}\) centered at zero. 
\end{lemma}

\begin{proof} Let $\gamma\in\CC$ and assume that \(F(z) \) takes the value \(\gamma \) for no \(z \in\mathbb D \). Since \(\left | F(z) \right | \le 1 + \sum_{j = 2}^N{\left |{{a_j}}\right |}\) for all \(z\in\DD\),  such values exist. Then the polynomial \(-\gamma + F(z) \) does not vanish for \(z \in \mathbb D \), and the reciprocal polynomial \({z^N}\left ({- \gamma + F\left ({{z^{- 1}}}\right)}\right) = - \gamma{z^N}+{z^{N - 1}}+{a_2}{z^{N - 2}}+{a_N}\) is Schur stable. The Vieta theorem now implies  \(\left |{\frac 1{\gamma}}\right| \le N  \), hence \(\left | \gamma \right | \ge \frac{1}{N}\). 
\end{proof}

\begin{remark}
The Vieta theorem implies that \(\left|\frac{a_j}{\gamma}\right| \leq \binom{N}{j}\), \(j = 1, \ldots, N\), therefore \(\sum_{j = 1}^N{\left |{\frac{{{a_j}}}{\gamma}}\right |}\le{2^N}- 1 \) and \(\sum_{j = 1}^N{\left |{{a_j}}\right |}\le \left ({{2^N}- 1}\right) \left | \gamma \right | \).
\end{remark}

The following lemma about factorization of  conjugate trigonometrical polynomials can be considered as a real analogue of the B\'ezout theorem. 

\begin{lemma} Let 
\begin{equation}\label{4}
C (t) = \sum_{j = 1}^N{{a_j}\cos jt}, \quad S(t) = \sum_{j = 1}^N{{a_j}\sin jt} 
\end{equation}
be a couple of conjugate trigonometrical polynomials with real coefficients. Suppose that 
\begin{equation}\label{5}
S ({t_1}) = \ldots = S ({t_m}) = 0, \quad C ({t_1}) = \ldots = C ({t_m}) = \gamma, 
\end{equation}
for some \({t_1}, \ldots,{t_m} \in \left ({0, \pi}\right) \) with \(2m \le n \). 
Then the trigonometrical polynomials \eqref{4}  have a representation of the form
\begin{align*}
C (t) & = \gamma + \prod_{k = 1}^m (\cos t - \cos{t_k})\cdot \sum_{j = m}^{N - m}{\alpha_j}\cos jt, \\
S(t) & = \prod_{k = 1}^m{(\cos t - \cos{t_k})}\cdot \sum_{j = m}^{N - m}{{\alpha_j}}\sin jt, 
\end{align*}
where the coefficients \({\alpha_m}, \ldots,{\alpha_{N - m}}\) are uniquely determined by \(\gamma,{a_1}, \ldots,{a_N}\), and \({\alpha_m}= -{2^m}\gamma \),  
\({a_1}= \frac{{{\alpha_{m + 1}}}}{{{2^m}}}- \frac{{{\alpha_m}}}{{{2^{m - 1}}}}\sum_{j = 1}^m{\cos{t_j}}\).
\end{lemma}

\begin{proof}
Consider the algebraic polynomial
\[
F\left( z \right) = - \gamma + \sum_{j = 1}^n{{a_j}{z^j}}.
\]
Then 
\(C(t) = \gamma + \Re\left(F\left(e^{it}\right)\right)\) and \(S(t) = \Im \left(F\left(e^{it}\right)\right)\). Because of \eqref{5}, \(F\left({{e^{i{t_j}}}}\right) = F\left(e^{- i{t_j}}\right) = 0\) for \(j = 1, \ldots, m\). 
By the Fundamental Theorem of Algebra there are numbers \({\beta_1},\ldots,{\beta_{n - 2m}}\) such that
\[
F\left(z\right) = \left(\prod_{k=1}^m\left(z -e^{i{t_k}}\right)\left(z -e^{-i{t_k}}\right)\right) \left({- \gamma + \sum_{j = 1}^{n - 2m}{\beta_j}{z^j}}\right).
\]
Let us write
\begin{align*}
\prod_{j = 1}^m\left({z -{e^{i{t_j}}}}\right)\left(z - e^{- i t_j}\right)&= \prod_{j = 1}^m \left({z^2- 2z\cos{t_j}+ 1}\right) \\
&= {2^m}{z^m}\prod_{j = 1}^m{\left({\frac{1}{2}\left(z + \frac{1}{z}\right) - \cos{t_j}}\right)}.
\end{align*}
Then
\[
F\left( z \right) = \prod_{k = 1}^m \left( \frac{1}{2}\left({z + \frac{1}{z}}\right) - \cos{t_k} \right) \left( -{2^m}\gamma{z^m}+ {2^m}\sum_{j = 1}^{n - 2m} \beta_j z^{m+j} \right).
\]
Therefore
\begin{align*}
\Re\left(F\left(e^{it}\right) \right)&= \prod_{k=1}^m\left(\cos t - \cos{t_k}\right)
\left( -{2^m}\gamma \cos mt + {2^m}\sum_{j=m+1}^{n-m}{\beta_{j-m}}\cos jt \right),\\
\Im \left({F\left({{e^{it}}}\right)}\right)&= \prod_{k = 1}^m \left(\cos t - \cos{t_k}\right) \left({-{2^m}\gamma \sin mt + 2^m\sum_{j = m + 1}^{n - m}{{\beta_{j-m}}\sin jt}}\right),
\end{align*}
which implies that
\begin{align*}
C(t) &= \gamma + \prod_{k = 1}^m{(\cos t - \cos{t_k})}\cdot \sum_{j = m}^{N - m}{{\alpha_j}}\cos jt,\\
S(t) &= \prod_{k = 1}^m{(\cos t - \cos{t_k})} \cdot \sum_{j = m}^{N - m}{{\alpha_j}}\sin jt
\end{align*}
with \(\alpha_m = -{2^m}\gamma\) and \( \alpha_{m + j}={2^m}{\beta_j}\) for \( j = 1, \ldots, n - 2m\).

Further, 
\begin{align*}
F'\left( z \right) &= \frac{d}{{dz}}\left({\prod_{j = 1}^m{\left({{z^2}- 2z\cos{t_j}+ 1}\right)}\left({- \gamma  + \sum_{j = 1}^{n - 2m}{{\beta_j}{z^j}}}\right)}\right) \\
&= 2\left({- \gamma  + \sum_{j = 1}^{n - 2m}{{\beta_j}{z^j}}}\right)\sum_{k = 1}^m \left({z - \cos{t_k}}\right) \prod_{\substack{j = 1\\j \ne k}}^m \left(z^2- 2z\cos t_j + 1\right) \\
&\quad + \prod_{j = 1}^m{\left({{z^2}- 2z\cos{t_j}+ 1}\right)}\sum_{j = 1}^{n - 2m}{j{\beta_j}{z^{j - 1}}}.
\end{align*}
Hence \(F'\left( 0 \right) = - 2\gamma \sum_{k = 1}^m{\cos t_k + {\beta_1}= - \frac{{{\alpha_m}}}{{{2^m - 1}}}} \sum_{k = 1}^m{\cos t_k}+ \frac{{{\alpha_{m + 1}}}}{{{2^m}}}\), which proves the lemma.
\end{proof}

We will need some properties of  Chebyshev polynomials of the second kind.

\begin{lemma} The following identities are valid:
\begin{align*}
 {U_{j - 1}}(x){U_{j + k}}(x) &={U_j}(x){U_{j + k - 1}}(x) -{U_{k - 1}}(x),\\
{U_{j + k}}(x) &={U_j}(x){U_k}(x) -{U_{j - 1}}(x){U_{k - 1}}(x),\\
\frac{d}{{dx}}{U_k}(x) &= \frac{1}{2(1 -x^2)}((k + 2)U_{k - 1}(x) - k{U_{k + 1}}(x))\\ 
&= \frac{1}{1 -x^2}((k + 1) U_{k - 1}(x) - kx U_k(x)).
\end{align*}
\end{lemma}

\begin{proof}This follows from the definition of Chebyshev polynomials by straightforward computations.
\end{proof}

Now, let us define two  \(N\times N\) matrices:
\[
A = \begin{pmatrix}
0 & 1/2 & 0 & 0 & \ldots\\
1/2 & 0 & 1/2 & 0 & \ldots\\
0 & 1/2 & 0 & 1/2 &\ldots\\
0 & 0 & 1/2 & 0 & \ldots\\
\ldots &\ldots &\ldots &\ldots &\ldots
\end{pmatrix}, 
\ B = \left({\begin{array}{*{20}{c}}
0&0&{{1 \mathord{\left/{\vphantom{1 2}}\right. \kern-\nulldelimiterspace}2}}&0& \ldots \\
0&0&0&{{1 \mathord{\left/{\vphantom{1 2}}\right. \kern-\nulldelimiterspace}2}}& \ldots \\
{{1 \mathord{\left/{\vphantom{1 2}}\right. \kern-\nulldelimiterspace}2}}&0&0&0& \ldots \\
0&{{1 \mathord{\left/{\vphantom{1 2}}\right. \kern-\nulldelimiterspace}2}}&0&0& \ldots \\
\ldots & \ldots & \ldots & \ldots & \ldots
\end{array}}\right).
\]
Let \(I\) denote the unit matrix.

\begin{lemma}
\(\det \left({4{x^2}(I - A) - (I - B)}\right) = \frac{1}{{{2^{N + 2}}x}}{U_{N + 1}}(x){U'_{N + 1}}(x)\).
\end{lemma}

\begin{proof} Set
\[
{\Phi_N}= 4{x^2}(I - A) - (I - B) = \left({\begin{array}{*{20}{c}}
{{U_2}}&{- x{U_1}}&{{1 \mathord{\left/{\vphantom{1 2}}\right. \kern-\nulldelimiterspace}2}}&0&0& \ldots \\
{- x{U_1}}&{{U_2}}&{- x{U_1}}&{{1 \mathord{\left/{\vphantom{1 2}}\right. \kern-\nulldelimiterspace}2}}&0& \ldots \\
{{1 \mathord{\left/{\vphantom{1 2}}\right. \kern-\nulldelimiterspace}2}}&{- x{U_1}}&{{U_2}}&{- x{U_1}}&{{1 \mathord{\left/{\vphantom{1 2}}\right. \kern-\nulldelimiterspace}2}}& \ldots \\
\ldots & \ldots & \ldots & \ldots & \ldots & \ldots 
\end{array}}\right), 
\]
where \({U_1}= 2x\) and \({U_2}= 4{x^2}- 1\) are the Chebyshev polynomials of the second kind of 
degree 1 and 2 respectively (we drop the variable in the polynomials for brevity).

Let 
\[
{\bar \Phi_N}= \left({\begin{array}{*{20}{c}}
{- x{U_1}}&\vline&{- x{U_1}}&{{1 \mathord{\left/{\vphantom{1 2}}\right. \kern-\nulldelimiterspace}2}}&0& \ldots \\
\hline
{{1 \mathord{\left/{\vphantom{1 2}}\right. \kern-\nulldelimiterspace}2}}&\vline&{}&{}&{}&{}\\
0&\vline&{}&{{\Phi_{N - 1}}}&{}&{}\\
 \ldots &\vline&{}&{}&{}&{}
\end{array}}\right).
\]
We compute the determinant of \({\Phi_N}\) by expanding along the first row:
\begin{multline}\label{6}
\left|{{\Phi_N}}\right| ={U_2}\left|{{\Phi_{N - 1}}}\right| + x{U_1}\left|{{{\bar \Phi}_{N - 1}}}\right| \\
- \frac{1}{2}\left({x{U_1}\left|{{{\bar \Phi}_{N - 2}}}\right| + \frac{1}{2}{U_2}\left|{{\Phi_{N - 3}}}\right| - \frac{1}{8}\left|{{\Phi_{N - 4}}}\right|}\right).
\end{multline}
Similarly, 
\begin{equation}\label{7}
\left|{{{\bar \Phi}_N}}\right| =  - x{U_1}\left|{{\Phi_{N - 1}}}\right| + \frac{x}{2}{U_1}\left|{{\Phi_{N - 2}}}\right| + \frac{1}{4}\left|{{{\bar \Phi}_{N - 2}}}\right|.
\end{equation}
Now,  write \eqref{6} for \(\left|{{\Phi_{N + 1}}}\right|\) and substitute \(\left|{{{\bar \Phi}_N}}\right|\) from \eqref{7}:
\begin{multline*}
\left|{{\Phi_{N + 1}}}\right| ={U_2}\left|{{\Phi_N}}\right| -{x^2}U_1^2\left|{{\Phi_{N - 1}}}\right| + \frac{{{x^2}}}{2}U_1^2\left|{{\Phi_{N - 2}}}\right| \\
- \frac{1}{4}{U_2}\left|{{\Phi_{N - 2}}}\right| + \frac{1}{{16}}\left|{{\Phi_{N - 3}}}\right| - \frac{x}{2}{U_1}\left|{{{\bar \Phi}_{N - 1}}}\right| + \frac{x}{4}{U_1}\left|{{{\bar \Phi}_{N - 2}}}\right|.
\end{multline*}
Taking into account \eqref{6}, we will  exclude \(\left|{{{\bar \Phi}_{N - 1}}}\right|\) and \(\left|{{{\bar \Phi}_{N - 2}}}\right|\) from the last equation;
note that the structure of  \eqref{6} allows us to exclude the whole linear combination  \( - \frac{x}{2}{U_1}\left|{{{\bar \Phi}_{N - 1}}}\right| + \frac{x}{4}{U_1}\left|{{{\bar \Phi}_{N - 2}}}\right|\) at one go, which is crucial:
\[
\begin{array}{c}
\left|{{\Phi_{N + 1}}}\right| = \left({{U_2}- \frac{1}{2}}\right)\left|{{\Phi_N}}\right| + \left({\frac{1}{2}{U_2}-{x^2}U_1^2}\right)\left|{{\Phi_{N - 1}}}\right|\\
+ \left({\frac{{{x^2}}}{2}U_1^2 - \frac{1}{4}{U_2}}\right)\left|{{\Phi_{N - 2}}}\right| + \left({\frac{1}{{16}}- \frac{1}{8}{U_2}}\right)\left|{{\Phi_{N - 3}}}\right| + \frac{1}{{32}}\left|{{\Phi_{N - 4}}}\right|.
\end{array}\]
Note that \({U_2}- \frac{1}{2}= 4{x^2}- \frac{3}{2}\) and \(\frac{1}{2}{U_2}-{x^2}U_1^2 = - 4{x^4}+ 2{x^2}- \frac{1}{2}\). Consider the difference equation
\begin{equation}\label{8}
\tiny
{Z_{N + 1}}= \left({4{x^2}- \frac{3}{2}}\right)\left({{Z_N}- \frac{1}{8}{Z_{N - 3}}}\right) 
+ \left({- 4{x^4}+ 2{x^2}- \frac{1}{2}}\right)\left({{Z_{N - 1}}- \frac{1}{2}{Z_{N - 2}}}\right) + \frac{1}{{32}}{Z_{N - 4}}.
\end{equation}
It is a linear equation with constant coefficients (with respect to \(N\)) of the fifth order. For any initial values of \(Z_1,\ldots,Z_5\) there exists a unique solution \({Z_N}\) of \eqref{8} for \(N = 6, 7, 8, \ldots.\) Note that the function  \({Z_N}= \left|{{\Phi_N}}\right|\) satisfies \eqref{8}.

To prove the lemma, we have to show that the function \({\phi_N}= \frac{1}{{{2^{N + 2}}x}}{U_{N + 1}}(x){U'_{N + 1}}(x)\) satisfies  \eqref{8} and the initial conditions \({\phi_j}= \left|{{\Phi_j}}\right|\) for \(j = 1, 2, 3, 4, 5\).

We  apply Lemma 3 to write
\[
{\phi_N}= \frac{1}{{{2^{N + 3}}x(1 -{x^2})}}{U_{N + 1}}(x)\left({(N + 3){U_N}(x) - (N + 1){U_{N + 2}}(x)}\right).
\]
The equalities
\begin{equation}\label{9}
{\phi_j}= \left|{{\Phi_j}}\right|, \quad j = 1, 2, 3, 4, 5,
\end{equation}
can easily be checked directly.

Denote \({U_{N - 3}}= \xi \) and \({U_{N - 4}}= \eta \). By Lemma 3, 
\[
{U_{N + j}}= \xi{U_{j + 3}}- \eta{U_{j + 2}},\quad  j = -2, -1, 0, 1, 2, 3.\] 
After substituting \({\phi_N}\) in \eqref{8} we get 
\begin{align*}
& \frac{1}{8}{U_{N + 2}}\left({(N + 4){U_{N + 1}}- (N + 2){U_{N + 3}}}\right)\\
& = \left({{x^2}- \frac{3}{8}}\right){U_{N + 1}}\left({(N + 3){U_N}- (N + 1){U_{N + 2}}}\right)\\
& \quad + \left({- 2{x^4}+{x^2}- \frac{1}{4}}\right){U_N}\left({(N + 2){U_{N - 1}}- N{U_{N + 1}}}\right)\\
& \quad + \left({2{x^4}-{x^2}+ \frac{1}{4}}\right){U_{N - 1}}\left({(N + 1){U_{N - 2}}- (N - 1){U_N}}\right) \\
& \quad + \left({-{x^2}+ \frac{3}{8}}\right){U_{N - 2}}\left({N{U_{N - 3}}- (N - 2){U_{N - 1}}}\right)\\
&\quad + \frac{1}{8}{U_{N - 3}}\left({(N - 1){U_{N - 4}}- (N - 3){U_{N - 2}}}\right).
\end{align*}
Inserting the above formulas for \(U_{N+j}\), we can express everything through the parameters \(\xi \), \(\eta \) and the Chebyshev polynomials \({U_0}(x),\ldots,{U_5}(x)\).
As a result, the left and right sides become the same fourteenth degree  polynomial of the variables \(\xi \), \(\eta \), \(x\), \(N\):
\begin{align*}
&- ((256N + 512){x^{11}}- (640N + 1408){x^9}+ (576N + 1376){x^7}\\
&- (224N + 576){x^5} + (35N + 96){x^3}- (3/2N +9/2)x)\xi  \\
&- ((64N + 128){x^9}- (128N + 288){x^7}+ (84N + 208){x^5} \\
&- (20N + 54){x^3}+ (5/4N +7/2)x)\eta  + ((256N + 512){x^{10}}\\
&- (576N + 1280){x^8} + (448N + 1088){x^6}- (140N + 368){x^4}\\
&+ (15N + 42){x^2}- (1/4N +3/4))\xi \eta,
\end{align*}
which can be easily checked by computer.

Thus, \({\phi_N}\) satisfies  \eqref{8}. Because it satisfies the same initial conditions \eqref{9}, by the uniqueness of solution to the difference equation \eqref{8} with given initial conditions we  conclude that \({\phi_N}\equiv{\Phi_N}\). 
\end{proof}

\begin{corollary}\label{c1} 
The roots of the equation
\[
\det \left( 4{x^2}(I - A) - (I-B) \right) = 0
\]
are the real numbers 
\(\left\{{\pm{\mu_j}}\right\}_{j = 1}^{\left\lfloor{\frac{{N + 1}}{2}}\right\rfloor}\), 
\(\left\{{\pm{\nu_j}}\right\}_{j = 1}^{N - \left\lfloor{\frac{{N + 1}}{2}}\right\rfloor}\),
where \({\mu_j}= \cos \frac{{j\pi}}{{N + 2}}\), \({U'_{N + 1}}\left({{\nu_j}}\right) = 0\), which can be ordered as
\begin{gather*}
0 <{\mu_{\frac{{N + 1}}{2}}}<{\nu_{\frac{{N - 1}}{2}}}<  \ldots  <{\nu_1}<{\mu_1}\quad\text{if  \(N\) is odd}, 
\\
0 <{\nu_{\frac{N}{2}}}<{\mu_{\frac{N}{2}}}<  \ldots  <{\nu_1}<{\mu_1} \quad\text{if \(N\) is even}.
\end{gather*}
\end{corollary}

The corollary follows from the fact that the zeros of the polynomials \({U_{N + 1}}(x)\) and \({U'_{N + 1}}(x)\) are alternating.

\begin{lemma} Let  \(A\) and \(B\) be as in the previous lemma. 
The one-parameter family \[c{\delta ^{(0)}}\left({\cos \frac{{j\pi}}{{N + 2}}}\right),\] where \(c \in \mathbb R\), \({\delta ^{(0)}}(x) ={\left({{U_0}(x){U_1}(x), \ldots,{U_{N - 1}}(x){U_N}(x)}\right)^T}\), and \(T\) denotes transposition, 
is a solution to the system of linear equations in~\(\delta\),
\[
\left({4{{\cos}^2}\frac{{j\pi}}{N + 2}(I - A) - (I - B)}\right)\delta  = 0, \quad j = 1, \ldots, \left\lfloor{\frac{{N + 1}}{2}}\right\rfloor. 
\]
\end{lemma}

\begin{proof}Let \({\Phi_N}(x) = 4{x^2}(I - A) - (I - B)\), as in the proof of Lemma 4. Let us multiply this matrix by the vector \({\delta ^{(0)}}(x)\) and write the product in coordinates:
\begin{align*}
& \left({4{x^2}- 1}\right){U_0}(x){U_1}(x) - 2{x^2}{U_1}(x){U_2}(x) + \frac{1}{2}{U_2}(x){U_3}(x) = 0, \\
\begin{split}
& - 2{x^2}{U_0}(x){U_1}(x) + \left({4{x^2}- 1}\right){U_1}(x){U_2}(x) - 2{x^2}{U_2}(x){U_3}(x) \\
& + \frac{1}{2}{U_3}(x){U_4}(x) = 0,
\end{split}\\
& \frac{1}{2}{U_j}(x){U_{j + 1}}(x) - 2{x^2}{U_{j + 1}}(x){U_{j + 2}}(x) + \left({4{x^2}- 1}\right){U_{j + 2}}(x){U_{j + 3}}(x)\\
& - 2{x^2}{U_{j + 3}}(x){U_{j + 4}}(x) + \frac{1}{2}{U_{j + 4}}(x){U_{j + 5}}(x) = 0, \quad j = 0, \ldots, N - 5,\\
\begin{split}
& \frac{1}{2}{U_{N - 4}}(x){U_{N - 3}}(x) - 2{x^2}{U_{N - 3}}(x){U_{N - 2}}(x) \\
& + \left({4{x^2}- 1}\right){U_{N - 2}}(x){U_{N - 1}}(x) - 2{x^2}{U_{N - 1}}(x){U_N}(x) = 0, 
\end{split}\\
\begin{split}
& \frac{1}{2}{U_{N - 3}}(x){U_{N - 2}}(x) - 2{x^2}{U_{N - 2}}(x){U_{N - 1}}(x) \\
& + \left({4{x^2}- 1}\right){U_{N - 1}}(x){U_N}(x) = 0.
\end{split}
\end{align*}
Now let us make the substitutions
\[
{U_j}(x) \to \frac{1}{{2\sin \frac{\pi}{{N + 2}}}}\left({\xi^{j + 1} - \xi^{-(j + 1)}}\right)
\] 
and 
\[
x \to \cos \frac{\pi}{{N + 2}}= \frac{1}{2}\left({{\xi ^{j + 1}}+{\xi ^{- (j + 1)}}}\right).
\] 
The first \(N - 2\) equations are identities for any \(\xi \), in particular for \(\xi ={e^{i\frac{{j\pi}}{{N + 2}}}}\). The next to last and the last equations are equivalent to 
\begin{align*}
\left({\xi ^{2N + 2}- 1}\right) \left({1 -{\xi ^{2N + 4}}}\right) &= 0,\\
\left({- 2\xi^{2N + 2}}-{\xi^{2N}}+{\xi ^2}+ 2 \right)\left({1 -{\xi ^{2N + 4}}}\right) &= 0.
\end{align*}
Both these equations have roots \(\xi ={e^{i\frac{{j\pi}}{{N + 2}}}}\). 
\end{proof}

We will need the following trigonometric identity.

\begin{lemma}
For any \(k = 0, 1, \ldots, N - 1\),
\begin{multline*}
2\sum_{j = 1}^{N - k}{\sin \frac{{(j + 1)\pi}}{{N + 2}}\sin \frac{{(j + k)\pi}}{{N + 2}}}= \left({N - k - 1}\right)\sin \frac{{k\pi}}{{N + 2}}\sin \frac{\pi}{{N + 2}} \\
+ \frac{1}{2}\frac{{\cos \frac{\pi}{{N + 2}}}}{{\sin \frac{\pi}{{N + 2}}}}\left({\left({N - k + 3}\right)\sin \frac{{(k + 1)\pi}}{{N + 2}}- \left({N - k + 1}\right)\sin \frac{{(k - 1)\pi}}{{N + 2}}}\right).
\end{multline*}
\end{lemma}

\begin{proof} 
We have
\allowdisplaybreaks
\begin{align*}
& 2\sum_{j = 1}^{N - k}  \sin \frac{{(j + 1)\pi}}{{N + 2}}\sin \frac{{(j + k)\pi}}{{N + 2}} \\
& = \sum_{j = 1}^{N - k}\left({\cos \frac{{(k - 1)\pi}}{{N + 2}}- \cos \frac{{(2j + k + 1)\pi}}{{N + 2}}}\right) \\
& = \left({N - k}\right)\cos \frac{{(k - 1)\pi}}{{N + 2}} - \cos \frac{{(k + 1)\pi}}{{N + 2}}\sum_{j = 1}^{N - k}\cos\frac{{2j\pi}}{{N + 2}} \\
& +\sin \frac{{(k + 1)\pi}}{{N + 2}}\sum_{j = 1}^{N - k}{\sin \frac{{2j\pi}}{{N + 2}}} = \left({N - k}\right)\cos \frac{{(k - 1)\pi}}{{N + 2}} \\
& - \cos\frac{{(k + 1)\pi}}{{N + 2}} \left(-\cos^2\frac{{(N - k + 1)\pi}}{{N + 2}}+ \frac{\cos\frac{\pi}{N + 2}}{\sin\frac{\pi}{N + 2}} \sin\frac{(N - k + 1)\pi}{N + 2} \right. \\
& \times \left.\cos\frac{{(N - k + 1)\pi}}{{N + 2}}\right) + \sin \frac{{(k + 1)\pi}}{{N + 2}} \left(- \frac{{\cos \frac{{\pi}}{{N + 2}}}}{{\sin \frac{{\pi}}{{N + 2}}}}{{\cos}^2}\frac{{(N - k + 1)\pi}}{{N + 2}}\right. \\
& - \sin \frac{{(N - k + 1)\pi}}{{N + 2}}\cos \frac{{(N - k + 1)\pi}}{{N + 2}} + \frac{{{{\cos}^3}\frac{{\pi}}{{N + 2}}}}{{\sin \frac{{\pi}}{{N + 2}}}}+ \sin \frac{{\pi}}{{N + 2}}\\
& \times\left.\cos \frac{{\pi}}{{N + 2}}\right) = \left({N - k}\right)\cos \frac{{(k - 1)\pi}}{{N + 2}}+ \frac{1}{{\sin \frac{{\pi}}{{N + 2}}}}\left({\left({- \sin \frac{{(k + 1)\pi}}{{N + 2}}}\right.\cos \frac{{\pi}}{{N + 2}}}\right.\\
& \left.{+ \sin \frac{{\pi}}{{N + 2}}\cos \frac{{(k + 1)\pi}}{{N + 2}}}\right) + \sin \frac{{(k + 1)\pi}}{{N + 2}}\cos \frac{{(k + 1)\pi}}{{N + 2}}\left({\cos \frac{{\pi}}{{N + 2}}\cos \frac{{(k + 1)\pi}}{{N + 2}}}\right.  \\
& \left.{\left.{+ \sin \frac{{\pi}}{{N + 2}}\sin \frac{{(k + 1)\pi}}{{N + 2}}}\right) + \cos \frac{{\pi}}{{N + 2}}\sin \frac{{(k + 1)\pi}}{{N + 2}}}\right)  \\
&  = \left({N - k}\right)\cos \frac{{(k - 1)\pi}}{{N + 2}}+ \frac{1}{{\sin \frac{{\pi}}{{N + 2}}}}\left({\cos \frac{{(k + 1)\pi}}{{N + 2}}\sin \frac{{\pi}}{{N + 2}}+ \sin \frac{{(k + 1)\pi}}{{N + 2}}\cos \frac{{\pi}}{{N + 2}}}\right)  \\
&  = \left({N - k}\right)\cos \frac{{(k - 1)\pi}}{{N + 2}}+ \frac{1}{{\sin \frac{{\pi}}{{N + 2}}}}\sin \frac{{(k + 2)\pi}}{{N + 2}}.
\end{align*}
On the other hand,
\[\begin{array}{l}
\left({N - k - 1}\right)\sin \frac{{k\pi}}{{N + 2}}\sin \frac{{\pi}}{{N + 2}}+ \frac{1}{2}\frac{{\cos \frac{\pi}{{N + 2}}}}{{\sin \frac{{\pi}}{{N + 2}}}}\left({\left({N - k + 3}\right)\sin \frac{{(k + 1)\pi}}{{N + 2}}- \left({N - k + 1}\right)\sin \frac{{(k - 1)\pi}}{{N + 2}}}\right)  \\
 = \left({N - k}\right)\left({\sin \frac{{k\pi}}{{N + 2}}\sin \frac{{\pi}}{{N + 2}}+ \frac{{\cos \frac{\pi}{{N + 2}}\sin \frac{{(k + 1)\pi}}{{N + 2}}}}{{2\sin \frac{{\pi}}{{N + 2}}}}- \frac{{\cos \frac{\pi}{{N + 2}}\sin \frac{{(k - 1)\pi}}{{N + 2}}}}{{2\sin \frac{{\pi}}{{N + 2}}}}}\right)  \\
 + \left({- \sin \frac{{k\pi}}{{N + 2}}\sin \frac{{\pi}}{{N + 2}}+ \frac{{3\cos \frac{\pi}{{N + 2}}\sin \frac{{(k + 1)\pi}}{{N + 2}}}}{{2\sin \frac{{\pi}}{{N + 2}}}}- \frac{{\cos \frac{\pi}{{N + 2}}\sin \frac{{(k - 1)\pi}}{{N + 2}}}}{{2\sin \frac{{\pi}}{{N + 2}}}}}\right)  
\end{array}
\]
\[
\begin{array}{l}
= \left({N - k}\right)\left({\sin \frac{{k\pi}}{{N + 2}}\sin \frac{{\pi}}{{N + 2}}+ \frac{{\cos \frac{\pi}{{N + 2}}}}{{2\sin \frac{{\pi}}{{N + 2}}}}2\cos \frac{{k\pi}}{{N + 2}}\sin \frac{\pi}{{N + 2}}}\right)  \\
+ \left({- \sin \frac{{k\pi}}{{N + 2}}\sin \frac{{\pi}}{{N + 2}}+ \frac{{\cos \frac{\pi}{{N + 2}}\sin \frac{{(k + 1)\pi}}{{N + 2}}}}{{\sin \frac{{\pi}}{{N + 2}}}}+ \frac{{\cos \frac{\pi}{{N + 2}}}}{{2\sin \frac{{\pi}}{{N + 2}}}}2\cos \frac{{k\pi}}{{N + 2}}\sin \frac{\pi}{{N + 2}}}\right)  \\
= \left({N - k}\right)\cos \frac{{(k + 1)\pi}}{{N + 2}}+ \cos \frac{{(k + 1)\pi}}{{N + 2}}+ \frac{{\cos \frac{\pi}{{N + 2}}\sin \frac{{(k + 1)\pi}}{{N + 2}}}}{{\sin \frac{{\pi}}{{N + 2}}}} \\
= \left({N - k}\right)\cos \frac{{(k + 1)\pi}}{{N + 2}}+ \frac{{\sin \frac{{(k + 2)\pi}}{{N + 2}}}}{{\sin \frac{{\pi}}{{N + 2}}}}.
\end{array}
\]
These imply the identity in the lemma.
\end{proof}

\section{The main result}

Now, let us turn to  Problem \ref{p1}. Recall that the polynomials considered are of the form \(F\left( z \right) = z+\sum_{j = 2}^N{{a_j}{z^j}}\).

\begin{lemma}
If a trigonometric polynomial \(\Im \left({F\left({{e^{it}}}\right)}\right)\) has a zero in \(\left({0, \pi}\right)\), then
\[
\min_t \left\{{\Re\left({F\left({{e^{it}}}\right)}\right):\Im \left({F\left({{e^{it}}}\right)}\right) = 0}\right\}<{J_N}.
\]
\end{lemma}

\begin{proof} Set
\[
C(t) = \Re\left({F\left({{e^{it}}}\right)}\right) = \sum_{j = 1}^N{{a_j}\cos jt, \quad}S(t) = \Im \left({F\left({{e^{it}}}\right)}\right) = \sum_{j = 1}^N{{a_j}\sin jt,}
\]
where \({a_1}= 1\). Since the coefficients $a_j$ are real, we can restrict ourselves to \(t \in \left[{0, \pi}\right]\). 

Let \({\rho_N}(a_2, \ldots, a_N) = \min_{t \in [0, \pi]}\left\{{C(t): S (t) = 0} \right\}\). It is clear that  \(\sup_{a_j}{\rho_N}(a_2, \ldots, a_N)\) 
%
increases in \(N\), and by Lemma 1 it is bounded from above by \(-\frac{1}{N}\). In addition, $\rho_2(a_2)=-1/2$ if $a_2=1/2$ and  by Lemma 1 we have \({J_2}= - 1/2\). It follows from the comments to Lemma~1 that 
\[
\sum_{j = 1}^N{\left|{{a_j}}\right|}\le \frac{1}{2}\left({{2^N}- 1}\right).
\] 
So the supremum \(\sup_{{a_j}}{{\rho_N}(a_2, \ldots, a_N)}\) can be taken over a bounded set 
\({A_R}= \{{\left({a_2, \ldots, a_N}\right):\sum_{j = 2}^N{\left|{{\alpha_j}}\right| \le R}}\}\)
for some \(R>0\).

The function \({\rho_N}(a_2, \ldots, a_N)\) is continuous on  \({A_R}\) except at those points \(\left({{a_2}\ldots, a_{N}}\right)\) for which the minimum value of \(C(t)\) is achieved at a zero of \(S(t)\) where this function does not change sign. The lower limit of  \({\rho_N}(a_2, \ldots, a_N)\) at each discontinuity point is equal to the value of the function. This means that on  \({A_R}\) the function \({\rho_N}(a_2, \ldots, a_N)\) is lower semicontinuous.

Along with  \({\rho_N}(a_2, \ldots, a_N)\) we will consider the function
\[
{\hat \rho_N}(a_2, \ldots, a_N) =\min_{}\left\{{C(t):t \in T \cup \left\{\pi \right\}}\right\},
\]
where $T$ is the set of points in \((0, \pi )\) where  \(S(t)\) changes sign. 
As \(T \cup \left\{\pi \right\}\) is contained in the set of zeros of  \(S(t)\), we have
\[
\sup_{({a_2},\ldots, a_N) \in{A_R}} \rho_N(a_2, \ldots, a_N) \le \sup_{({a_2},\ldots, a_N) \in{A_R}} \hat \rho_N(a_2, \ldots, a_N).
\]
Since \({\hat \rho_N}(a_2, \ldots, a_N)\) is upper semicontinuous, it attains on \({A_R}\)  a maximal value
\[\hat \rho_N^{(0)}=\max_{({a_2},\ldots, a_N) \in{A_R}} \hat \rho_N(a_2, \ldots, a_N).
\] 
It follows from Lemma 1 that \(\hat \rho_N^{(0)}< 0\).

Let us call a pair of trigonometric polynomials \(\left\{{{C^{(0)}}(t),{S^{(0)}}(t)}\right\}\) at which the maximum is achieved an \emph{optimal pair.} Let us assume that for the optimal polynomial \({S^{(0)}}(t)\) the set \(T = \left\{t_1, \ldots, t_q\right\}\), where  
\(1 \le q \le N - 1\), is not empty. Assume moreover that
\[
\min \left\{{{C^{(0)}}({t_1}),\ldots,{C^{(0)}}(t_q)}\right\}=C^{(0)}({t_1}),
\]
and 
\begin{align*}
C^{(0)}(t_1) &={C^{(0)}}({t_j}), \quad j = 1,\ldots, m\ (1 \le m \le q), \\
C^{(0)}(t_1) &< C^{(0)}({t_j}), \quad j = m + 1, \ldots, q.
\end{align*}

Our goal is to show that  \(T\) is in fact empty. For this purpose, we construct an auxiliary class of polynomials with the first coefficient equal to~1, and  compare the value of  \({\hat \rho_N}(a_2, \ldots, a_N)\) for  \(\left\{{{C^{(0)}}(t),{S^{(0)}}(t)}\right\}\) with the value  for a pair of  conjugate trigonometric polynomials from the auxiliary class.

Two cases  are possible: 1. \({C^{(0)}}({t_1}) <{C^{(0)}}(\pi )\) or 2. \({C^{(0)}}(\pi ) \le {C^{(0)}}({t_1})\).\\

\noindent
Case 1. According to Lemma 2,
\begin{align*}
{S^{(0)}}(t) &= \prod_{k = 1}^m{(\cos t - \cos{t_k})}\cdot \sum_{j = m}^{N - m}{{\alpha_j}}\sin jt,
\\
{C^{(0)}}(t) &= - \frac{{{\alpha_m}}}{{{2^m}}}+ \prod_{k = 1}^m{(\cos t - \cos{t_k})}\cdot \sum_{j = m}^{N - m}{{\alpha_j}}\cos jt.
\end{align*}
As \({C^{(0)}}({t_1}) = - \frac{{{\alpha_m}}}{{{2^m}}}\) and \(\hat \rho_N{(0)}< 0\), we have \(\alpha_m^{(m)}> 0\). In addition,  \({a_1}= 1\), so by Lemma 2, 
\(\frac{{{\alpha_{m + 1}}}}{{{2^m}}}-\frac{{{\alpha_m}}}{{{2^m - 1}}}\sum_{j = 1}^m{\cos{t_j}}= 1\).

Let us construct an auxiliary class of trigonometric polynomials
\begin{align*}
S({\theta_1}\ldots,{\theta_m};t) &= N({\theta_1}\ldots,{\theta_m}) \cdot \prod_{k = 1}^m{(\cos t - \cos{\theta_k})}\sum_{j = m}^{N - m}{{\alpha_j}\sin jt},
\\
C({\theta_1}\ldots,{\theta_m};t) &= N({\theta_1}\ldots,{\theta_m}) \cdot \left({- \frac{{{\alpha_m}}}{{{2^m}}}+ \prod_{k = 1}^m{(\cos t - \cos{\theta_k})}\sum_{j = m}^{N - m}{{\alpha_j}\cos jt}}\right),
\end{align*}
where  \(N({\theta_1}\ldots,{\theta_m})\) ensures that the first coefficient of  \(S({\theta_1}\ldots,{\theta_m};t)\) and \(C({\theta_1}\ldots,{\theta_m};t)\) is 1:
\[
N({\theta_1}\ldots,{\theta_m}) = \frac{1}{\frac{\alpha_{m + 1}}{2^m} - \frac{\alpha_m}{2^m - 1}\sum_{j = 1}^m \cos{\theta_j}}.
\]
The polynomials \(C({\theta_1}\ldots,{\theta_m};t)\) and \(S({\theta_1}\ldots,{\theta_m};t)\) are conjugate and for \(\left\{{{\theta_1},\ldots,{\theta_m}}\right\}= \left\{{{t_1}\ldots,{t_m}}\right\}\) coincide with  \({C^{(0)}}(t)\) and \({S^{(0)}}(t)\) respectively. Note that the set of sign changes o for  \(S({\theta_1}\ldots,{\theta_m};t)\)  is \(T_{\theta}= \left\{\theta_1, \ldots, \theta_m, t_{m + 1}, \ldots, t_q\right\}\).

Finally,  \(S({\theta_1}\ldots,{\theta_m};t)\), \(C({\theta_1}\ldots,{\theta_m};t)\) can be written as
\begin{align*}
S({\theta_1}, \ldots,{\theta_m};t) &= \frac{{\prod_{k = 1}^m{(\cos t - \cos{\theta_k})}\sum_{j = m}^{N - m}{{\alpha_j} \sin jt}}}{{\frac{{{\alpha_{m + 1}}}}{{{2^m}}}- \frac{{{\alpha_m}}}{{{2^{m - 1}}}}\sum_{j = 1}^m{\cos{\theta_j}}}}, 
\\
C({\theta_1}, \ldots,{\theta_m};t) &= \frac{{- \frac{{{\alpha_m}}}{{{2^m}}}+ \prod_{k = 1}^m{(\cos t - \cos{\theta_k})}\sum_{j = m}^{N - m}{{\alpha_j}\cos jt}}}{{\frac{{{\alpha_{m + 1}}}}{{{2^m}}}- \frac{{{\alpha_m}}}{{{2^{m - 1}}}}\sum_{j = 1}^m{\cos{\theta_j}}}}.
\end{align*}

Let us show that the value \({\hat \rho_N}\) for 
\(\left\{{C({\theta_1}\ldots,{\theta_m};t), S({\theta_1}\ldots,{\theta_m};t)}\right\}\) is greater than for  \(\left\{{{C^0}(t),{S^0}(t)}\right\}\), so the pair \(\left\{{{C^0}(t),{S^0}(t)}\right\}\) is not  optimal.

It is clear that
\begin{multline*}
C({\theta_1}, \ldots,{\theta_m};{\theta_1}) =  \ldots  = C({\theta_1}, \ldots,{\theta_m};{\theta_m}) \\
= - \frac{{\frac{{{\alpha_m}}}{{{2^m}}}}}{{\frac{{{\alpha_{m + 1}}}}{{{2^m}}}- \frac{{{\alpha_m}}}{{{2^{m - 1}}}}\sum_{j = 1}^m{\cos{\theta_j}}}}=\frac{1}{{2\sum_{j = 1}^m{\cos{\theta_j}} - \frac{{{\alpha_{m + 1}}}}{{{a_m}}}}}.
\end{multline*}
If \(0 < \left|{{\theta_j}-{t_j}}\right| < \varepsilon\) for all \(j = 1, \ldots, m\), then the value of  \(\frac{1}{2\sum_{j = 1}^m \cos\theta_j - \frac{\alpha_m}{\alpha_{m - 1}}}\) is close to \( - \frac{{{\alpha_m}}}{{{2^m}}}={C^{(0)}}\left({{t_1}}\right)\).
 
Since the cosine function decreases on \(\left(0, \pi\right)\), the value of \(C({\theta_1}\ldots, {\theta_m}; \theta_1)\) increases in each of the parameters \(\theta_1, \ldots, \theta_m\). By continuity of trigonometric polynomials in \(t\) and in all  coefficients we have 
\begin{align*}
&C({\theta_1}\ldots,{\theta_m};{\theta_j}>{C^0}({t_j}),\quad j = 1, \ldots, m,
\\
&\left|{C({\theta_1}\ldots,{\theta_m};{t_j}) -{C^0}({t_j})}\right| < \delta,\quad j = m + 1, \ldots, q,
\\
&\left|{C({\theta_1}\ldots,{\theta_m};\pi ) -{C^0}(\pi )}\right| < \delta,
\end{align*}
for any arbitrarily small \(\delta \) for an appropriate choice of \(\varepsilon \). These inequalities mean that 
\begin{multline*}
\min \{C({\theta_1}, \ldots,{\theta_m};{\theta_1}), \ldots, C({\theta_1}, \ldots,{\theta_m};{\theta_m}), \\
C({\theta_1}, \ldots,{\theta_m};{t_{m + 1}}), \ldots, C({\theta_1}, \ldots,{\theta_m};{t_q}), \ldots, C({\theta_1}, \ldots,{\theta_m};\pi )\}
\end{multline*}
is larger than \(\min \left\{{{C^0}({t_1}),\ldots,{C^0}({t_q}),{C^0}(\pi )}\right\}\), at least for sufficiently small positive values of \({\theta_j}-{t_j}\), \(j = 1, \ldots, m\). Thus,  \(\left\{{{C^0}(t),{S^0}(t)}\right\}\) is not optimal.
\medskip

\noindent
Case 2. We have \(C^0(\pi ) = - \frac{{{\alpha_m}}}{{{2^m}}}- \prod_{k = 1}^m{(1 + \cos{t_k})}\sum_{j = m}^{N - m} (-1)^j\alpha_j\),
\[
C({\theta_1}, \ldots,{\theta_m};\pi ) =  - \frac{{\frac{{{\alpha_m}}}{{{2^m}}}+ \prod_{k = 1}^m {(1 + \cos{\theta_k})} \sum_{j = m}^{N - m}{{{( - 1)}^j}{\alpha_j}}}}{{\frac{{{\alpha_{m + 1}}}}{{{2^m}}}- \frac{{{\alpha_m}}}{{{2^{m - 1}}}}\sum_{j = 1}^m{\cos{\theta_j}}}}, 
\]
and \(C({t_1}\ldots,{t_m};\pi ) ={C^0}(\pi )\). Since by the assumptions \({C^0}(\pi ) \le  - \frac{{{\alpha_m}}}{{{2^m}}}\), it follows that \(\sum_{j = m}^{N - m}{{{( - 1)}^j}{\alpha_j}} \ge 0\), and the value 
\[\frac{{{\alpha_m}}}{{{2^m}}}+ \prod_{k = 1}^m{(1 + \cos{\theta_k})}\sum_{j = m}^{N - m}{{{( - 1)}^j}{\alpha_j}}\]
decreases in each parameter \({\theta_1},\ldots,{\theta_m}\). 
For small positive  \({\theta_j}-{t_j}\), \(j = 1, \ldots, m\) the value  
\(\frac{{{\alpha_{m + 1}}}}{{{2^m}}}- \frac{{{\alpha_m}}}{{{2^{m - 1}}}}\sum_{j = 1}^m{\cos{\theta_j}}\)
is close to 1 and increases in each parameter \({\theta_1},\ldots,{\theta_m}\). Thus,  \(C({\theta_1},\ldots,{\theta_m};\pi )\) increases in each  \({\theta_1},\ldots,{\theta_m}\) as well. Moreover, the values  
\(C({\theta_1},\ldots,{\theta_m};{\theta_j})\), \(j = 1, \ldots, m\), are all equal and also increase in    \({\theta_1},\ldots,{\theta_m}\). Hence, in this case too, the pair \(\left\{{{C^0}(t),{S^0}(t)}\right\}\) is not  optimal.

Thus, it is shown that the set $T$ is empty, i.e., the optimal pair \(\left\{{{C^0}(t),{S^0}(t)}\right\}\) satisfies \({S^0}(t) \ge 0\) for all \(t \in \left[{0, \pi}\right]\).
\end{proof}

\subsection{Remarks}
The idea of the proof  is the construction of an auxiliary set of polynomials 
with further variation of the original polynomial over this set. Suppose for  \({S^{(0)}}(t)\) the set \(T\) is a disjoint union \({T_1}\cup{T_2}\), where \({T_1}= \left\{{{t_1},\ldots,{t_m}}\right\}\), \({T_2}= \left\{{{t_{m + 1}},\ldots,{t_q}}\right\}\), and both sets are nonempty. The trigonometric polynomial \({S^{(0)}}(t)\) is determined by the parameters \({\alpha_m}, \ldots,{\alpha_{N - m}}\) and by the set \({T_1}\). 

Let us use the parameters \({\alpha_m}, \ldots,{\alpha_{N - m}}\)  to build an auxiliary set of polynomials.

First,  let \({\theta_j}={t_j}\), \(j = 1, \ldots, m\). Further, we will vary the parameters \({\theta_1},\ldots,{\theta_m}\) independently. 

If \({C^{(0)}}(\pi ) <{C^{}{(0)}}({t_1})\) then as a result of variation we can ensure that \({C^{(1)}}(\pi ) ={C^{(1)}}({t_1})\), where  \({C^{(1)}}(t)\) is obtained from  \({C^{(0)}}(t)\) by variation of parameters. Then we recalculate the set \({T_1}= \left\{{{t_{11}},\ldots,{t_{1{m_1}}}}\right\}\) for the polynomial \({S^{(1)}}(t)\). 

The parameters \({\alpha_m}, \ldots,{\alpha_{N - m}}\) remain the same because \({T_2}\) does not change. Then we construct a new auxiliary set. 

For the new set we again vary \({\theta_1},\ldots,{\theta_{{m_1}}}\) until we get  \({C^{(2)}}({t_s}) ={C^{(1)}}({t_{11}})\), \({t_s}\in{T_2}\). This means that \({T_2}\) has at least  one element less. For \({S^{(2)}}(t)\) we construct a new auxiliary set. By variation of parameters one can make  \({T_2}\) empty. In this case  \({T_1}\) will be either empty, or a singleton. In the latter case by variation over the auxiliary set of polynomials we achieve that \({T_1}\) is  empty. Finally, the set \(T\) will be empty.
\medskip

It follows from  Lemma 7 that the optimal polynomial is necessarily  
typically real. Recall that the function is typically real in $\mathbb D$ if it is real for at every real point of the disc and in all the others points of the disc we have $\Im\{f(z)\}\Im\{z\}>0.$ The set of all typically real polynomials of degree
\(\le N\) will be denoted by \({{\rm T}_N}\).

\begin{theorem}
We have 
\begin{align*}
J_N &=  - \frac{1}{4}{\sec ^2}\frac{\pi}{{N + 2}}, 
\\
a_j^{(0)}&= \frac{1}{{{{U'}_N}\left({\cos \frac{\pi}{{N + 2}}}\right)}}{U'_{N - j + 1}}\left({\cos \frac{\pi}{{N + 2}}}\right){U_{j - 1}}\left({\cos \frac{\pi}{{N + 2}}}\right).
\end{align*}
\end{theorem}

\begin{proof}
By Lemma 7,
\begin{align*}
J_N&= \sup_{{a_1}= 1}\left({\min_t \left\{{\Re\left({F\left({{e^{it}}}\right)}\right):\Im \left({F\left({{e^{it}}}\right)}\right) = 0}\right\}}\right)\\
&\leq
 \max_{{a_1}= 1}\left({\min_{t \in T \cup \left\{\pi \right\}}\left\{{\Re\left({F\left({{e^{it}}}\right)}\right)}\right\}}\right),
\end{align*}
where $T$ is the set of points in \((0, \pi )\) at which \(\Im \left({F\left({{e^{it}}}\right)}\right)\) changes sign. By Lemma 7 the maximum is achieved on the class \({{\rm T}_N}\), and if  \(F_N^{(0)}\left( z \right)\) is the optimal polynomial, then \({J_N}= F_N^{(0)}\left({- 1}\right)\).

Further, we have
\[\Im\left({F_{N}^{(0)}\left({{e^{it}}}\right)}\right) = \sin t + \sum_{j = 2}^N{{\alpha_j}\sin jt} = {\beta_0}\sin t\left({1 + 2\sum_{j = 1}^{N - 1}{{\beta_j}\cos jt}}\right),
\]
where \({\beta_0}= 1 + \sum_{j = 1}^{\left\lfloor{\frac{{N - 1}}{2}}\right\rfloor}{{\alpha_{2j + 1}}}\), 
\({\beta_1}= \frac{1}{{{\beta_0}}}\sum_{j = 1}^{\left\lfloor{\frac{N}{2}}\right\rfloor}{{\alpha_{2j}}}\), \({\beta_2}= \frac{1}{{{\beta_0}}}\sum_{j = 1}^{\left\lfloor{\frac{{N - 1}}{2}}\right\rfloor}{{\alpha_{2j + 1}}}\), $\ldots$, \({\beta_{N - 1}}= \frac{1}{{{\beta_0}}}{\alpha_N}\).
By Lemma 7 the trigonometric polynomial \(1 + 2\sum_{j = 1}^{N - 1}{{\beta_j}\cos jt}\) is nonnegative on \(\left[0, \pi\right]\), and therefore
\begin{align*}
{J_N}&= F_N^{(0)}\left({- 1}\right) = \sup_{F \in{{\rm T}_N}}\left\{{{F_N}\left({- 1}\right)}\right\}\\
&=\sup_{{\alpha_j}}\left\{{- 1 +{\alpha_2}-{\alpha_3}+  \ldots :1 + 2\sum_{j = 1}^{N - 1}{{\beta_j}\cos jt} \ge 0}\right\}\\
&= \sup_{{\beta_j}}\left\{{- \frac{{1 -{\beta_1}}}{{1 -{\beta_2}}}:1 + 2\sum_{j = 1}^{N - 1}{{\beta_j}\cos jt} \ge 0}\right\}. 
\end{align*}

By the Fej\'er-Riesz theorem \cite[6.5,  problem 41]{15.} any nonnegative trigonometric polynomial can be represented as the square of the modulus of some polynomial: 
\[
{\left|{{d_1}+{d_2}{e^{it}}+  \ldots  +{d_N}{e^{i(N - 1)t}}}\right|^2}= 1 + 2\sum_{j = 1}^{N - 1}{{\beta_j}\cos jt}.
\] 
Consequently, 
\begin{align}\label{10}
&d_1^2 + d_2^2 + \ldots + d_n^2 = 1,
\\
\label{11}
&{d_1}{d_2}+{d_2}{d_3}+ \ldots +{d_{N - 1}}{d_N}={\beta_1},
\\
\label{12}
&{d_1}{d_3}+{d_2}{d_4}+ \ldots +{d_{N - 2}}{d_N}={\beta_2}, 
\\ \notag
&\ldots
\\ \notag
&{d_1}{d_N}={\beta_{N - 1}}.
\end{align}
Then we can write 
\({J_N}=\max_{{d_j}}\left\{{- \frac{{1 -{d^T}Ad}}{{1 -{d^T}Bd}}:{d^T}d = 1}\right\}\),
 where
\(d ={\left({{d_1}\ldots,{d_N}}\right)^T}\), and
\begin{align*}
A &= \left({\begin{array}{*{20}{c}}
0&{{1 \mathord{\left/{\vphantom{1 2}}\right. \kern-\nulldelimiterspace}2}}&0&0& \ldots \\
{{1 \mathord{\left/{\vphantom{1 2}}\right. \kern-\nulldelimiterspace}2}}&0&{{1 \mathord{\left/{\vphantom{1 2}}\right. \kern-\nulldelimiterspace}2}}&0& \ldots \\
0&{{1 \mathord{\left/{\vphantom{1 2}}\right. \kern-\nulldelimiterspace}2}}&0&{{1 \mathord{\left/{\vphantom{1 2}}\right. \kern-\nulldelimiterspace}2}}& \ldots \\
0&0&{{1 \mathord{\left/{\vphantom{1 2}}\right. \kern-\nulldelimiterspace}2}}&0& \ldots \\
 \ldots & \ldots & \ldots & \ldots & \ldots 
\end{array}}\right),
\\
B &= \left({\begin{array}{*{20}{c}}
0&0&{{1 \mathord{\left/{\vphantom{1 2}}\right. \kern-\nulldelimiterspace}2}}&0& \ldots \\
0&0&0&{{1 \mathord{\left/{\vphantom{1 2}}\right. \kern-\nulldelimiterspace}2}}& \ldots \\
{{1 \mathord{\left/{\vphantom{1 2}}\right. \kern-\nulldelimiterspace}2}}&0&0&0& \ldots \\
0&{{1 \mathord{\left/{\vphantom{1 2}}\right. \kern-\nulldelimiterspace}2}}&0&0& \ldots \\
 \ldots & \ldots & \ldots & \ldots & \ldots \end{array}}\right)
\end{align*}
are the \(N\times N\) matrices of the quadratic forms \eqref{11} and \eqref{12}, respectively. Consequently, 
\begin{align*}
{J_N} &=\max_{{d_j}}\left\{{- \frac{{1 - \frac{{{d^T}Ad}}{{{d^T}d}}}}{{1 - \frac{{{d^T}Bd}}{{{d^T}d}}}}}\right\}=\max_{{d_j}}\left\{{- \frac{{{d^T}(I - A)d}}{{{d^T}(I - B)d}}}\right\}\\ 
 &= -\min_{{d_j}}\left\{{\frac{{{d^T}(I - A)d}}{{{d^T}(I - B)d}}}\right\},
\end{align*} 
where \(I\) is the unit matrix. It is clear that the matrices  \(I - A\) and \(I - B\) are positive definite.

The problem of finding \({J_N}= - \min_{{d_j}}\left\{{\frac{{{d^T}(I - A)d}}{{{d^T}(I - B)d}}}\right\}\) reduces to
computing generalized eigenvalues \cite{16.,17.}: if \({\lambda_1}\le \cdots \le{\lambda_N}\) are the roots of the equation
\[
\det \left({(I - A) - \lambda (I-B)}\right) = 0,
\]
then \({J_N}= -{\lambda_1}\). Because  \(I - A\) and \(I - B\) are positive definite, we have \({\lambda_1}> 0\). The corresponding minimum is achieved on the generalized eigenvector \({\delta ^{(0)}}\), which is determined from the relation \((I - A){\delta ^{(0)}}={\lambda_1}(I - B){\delta ^{(0)}}\).

By Corollary \ref{c1},
\[
{J_N}= -{\lambda_1}= - \frac{1}{{4{\cos^2}\frac{\pi}{{N + 2}}}},
\]
and by Lemma 5,
\begin{multline*}
{\delta ^{(0)}}= c\left(U_0\left({\cos \frac{\pi}{{N + 2}}}\right){U_1}\left({\cos \frac{\pi}{{N + 2}}}\right), \ldots,\right. \\
\left.{U_{N - 1}}\left({\cos \frac{\pi}{{N + 2}}}\right){U_N}\left({\cos \frac{\pi}{{N + 2}}}\right)\right)^T.
\end{multline*}

The next step is to find the coefficients \(a_j^{(0)}\), \(j = 1, \ldots, N\). Note that the scaling constant $c$ appears as a factor to both numerator and denominator in formula for $a_j^{(0)}$ below, thus without loss of generality we can take $c=1.$

We have the following relations between \(\delta_j^{(0)}\), \(\beta_j^{(0)}\) and \(a_j^{(0)}\):
\begin{align*}
\beta_j^{(0)}&= \sum_{k = 1}^{N - j}{\delta_k^{(0)}}\delta_{k + j}^{(0)},\quad j = 0, \ldots, N - 1,\\
\beta_N^{(0)}&= \beta_{N + 1}^{(0)}= 0,
\\
a_j^{(0)}&= \frac{1}{{\beta_0^{(0)}- \beta_2^{(0)}}}\left({\beta_{j - 1}^{(0)}- \beta_{j + 1}^{(0)}}\right),\quad j = 1, \ldots, N, 
\end{align*}
where \(\delta_j^{(0)}={U_{j - 1}}\left({\cos \frac{\pi}{{N + 2}}}\right){U_j}\left({\cos \frac{\pi}{{N + 2}}}\right)\), \(j = 1, \ldots, N\).

Denote \({x_0}= \cos \frac{\pi}{{N + 2}}\). Then
\[
\beta_0^{(0)}- \beta_2^{(0)}= \sum_{j = 1}^N{{{\left({{U_{j - 1}}\left({{x_0}}\right){U_j}\left({{x_0}}\right)}\right)}^2}} - \sum_{j = 1}^{N - 2}{{U_{j - 1}}\left({{x_0}}\right){U_j}\left({{x_0}}\right){U_{j + 1}}\left({{x_0}}\right){U_{j + 2}}\left({{x_0}}\right)}.
\]
Lemma 3 implies that 
\begin{align*}
{U_{j - 1}}\left({{x_0}}\right){U_{j + 1}}\left({{x_0}}\right) &={\left({{U_j}\left({{x_0}}\right)}\right)^2}- 1,\\
{U_j}\left({{x_0}}\right){U_{j + 2}}\left({{x_0}}\right) &={\left({{U_{j + 1}}\left({{x_0}}\right)}\right)^2}- 1.
\end{align*} 
Taking into account that \({U_{N - j}}\left({{x_0}}\right) ={U_j}\left({{x_0}}\right)\) we get 
\begin{gather*}
\beta_0^{(0)}- \beta_2^{(0)}={\left({{U_0}\left({{x_0}}\right){U_1}\left({{x_0}}\right)}\right)^2}\\
+ \sum_{j = 1}^{N - 2}{{{\left({{U_{j - 1}}\left({{x_0}}\right){U_j}\left({{x_0}}\right)}\right)}^2}} +{\left({{U_{N - 1}}\left({{x_0}}\right){U_N}\left({{x_0}}\right)}\right)^2}\\
-\sum_{j = 1}^{N - 2}{\left({{{\left({{U_j}\left({{x_0}}\right){U_{j + 1}}\left({{x_0}}\right)}\right)}^2}-{{\left({{U_j}\left({{x_0}}\right)}\right)}^2}-{{\left({{U_{j + 1}}\left({{x_0}}\right)}\right)}^2}+ 1}\right)} \\
= 2\sum_{j = 1}^{N - 1}{{{\left({{U_j}\left({{x_0}}\right)}\right)}^2}} - (N - 2) \\
= \frac{1}{{{{\sin}^2}\frac{\pi}{{N + 2}}}}\left({2\sum_{j = 1}^{N - 1}{{{\sin}^2}\frac{{\pi (j + 1)}}{{N + 2}}} - (N - 2){{\sin}^2}\frac{\pi}{{N + 2}}}\right) \\
= \frac{1}{{{{\sin}^2}\frac{\pi}{{N + 2}}}}\left({N + 2\cos \frac{{2\pi}}{{N + 2}}- (N - 2){{\sin}^2}\frac{\pi}{{N + 2}}}\right) = \left({N + 2}\right){\cot ^2}\frac{\pi}{{N + 2}}.
\end{gather*}
Thus,
\[
\beta_0{(0)}- \beta_2{(0)}= \left({N + 2}\right)\frac{{x_0^2}}{{1 - x_0^2}}.
\]
Similarly
\begin{multline*}
\beta_{k - 1}^{(0)}- \beta_{k + 1}^{(0)}= \sum_{j = 1}^{N - k + 1}{{U_{j - 1}}\left({{x_0}}\right){U_j}\left({{x_0}}\right){U_{j + k - 2}}\left({{x_0}}\right){U_{j + k - 1}}\left({{x_0}}\right)}\\
- \sum_{j = 1}^{N - k - 1}{{U_{j - 1}}\left({{x_0}}\right){U_j}\left({{x_0}}\right){U_{j + k}}\left({{x_0}}\right){U_{j + k + 1}}\left({{x_0}}\right)}.
\end{multline*}

Lemma 3 implies that 
\begin{align*}
{U_{j - 1}}\left({{x_0}}\right){U_{j + k}}\left({{x_0}}\right) &={U_j}\left({{x_0}}\right){U_{j + k - 1}}\left({{x_0}}\right) -{U_{k - 1}}\left({{x_0}}\right),
\\
{U_j}\left({{x_0}}\right){U_{j + k + 1}}\left({{x_0}}\right) &={U_{j + 1}}\left({{x_0}}\right){U_{j + k}}\left({{x_0}}\right) -{U_{k - 1}}\left({{x_0}}\right).
\end{align*} 
Then 
\begin{gather*}
\beta_{k - 1}^{(0)}- \beta_{k + 1}^{(0)}={U_0}\left({{x_0}}\right){U_1}\left({{x_0}}\right){U_{k - 1}}\left({{x_0}}\right){U_k}\left({{x_0}}\right) \\
+ \sum_{j = 2}^{N - k - 1}{{U_{j - 1}}\left({{x_0}}\right){U_j}\left({{x_0}}\right){U_{j + k - 2}}\left({{x_0}}\right){U_{j + k - 1}}\left({{x_0}}\right)}  \\
+{U_{N - k}}\left({{x_0}}\right){U_{N - k + 1}}\left({{x_0}}\right){U_{N - 1}}\left({{x_0}}\right){U_N}\left({{x_0}}\right)  \\
\sum_{j = 1}^{N - k - 1}{{U_j}\left({{x_0}}\right){U_{j + 1}}\left({{x_0}}\right){U_{j + k - 1}}\left({{x_0}}\right){U_{j + k}}\left({{x_0}}\right)}  \\
+ \sum_{j = 1}^{N - k - 1}{{U_j}\left({{x_0}}\right){U_{j + k - 1}}\left({{x_0}}\right){U_{k - 1}}\left({{x_0}}\right)} \\
+ \sum_{j = 1}^{N - k - 1}{{U_{j + 1}}\left({{x_0}}\right){U_{j + k}}\left({{x_0}}\right){U_{k - 1}}\left({{x_0}}\right)} - (N - k - 1){\left({{U_{k - 1}}\left({{x_0}}\right)}\right)^2}\\
={U_1}\left({{x_0}}\right){U_{k - 1}}\left({{x_0}}\right){U_k}\left({{x_0}}\right) + \sum_{j = 1}^{N - k - 1}{{U_{j + 1}}\left({{x_0}}\right){U_{j + k}}\left({{x_0}}\right){U_{k - 1}}\left({{x_0}}\right)} \\
+ \sum_{j = 1}^{N - k - 1}{{U_j}\left({{x_0}}\right){U_{j + k - 1}}\left({{x_0}}\right){U_{k - 1}}\left({{x_0}}\right)}  \\
+{U_{N - k}}\left({{x_0}}\right){U_{N - k + 1}}\left({{x_0}}\right){U_N}\left({{x_0}}\right) - (N - k - 1){\left({{U_{k - 1}}\left({{x_0}}\right)}\right)^2} \\
= 2{U_{k - 1}}\left({{x_0}}\right)\sum_{j = 1}^{N - k}{{U_j}\left({{x_0}}\right){U_{j + k - 1}}\left({{x_0}}\right)} - (N - k - 1){\left({{U_{k - 1}}\left({{x_0}}\right)}\right)^2}\\
= \frac{{{U_{k - 1}}\left({{x_0}}\right)}}{{{{\sin}^2}\frac{\pi}{{N + 2}}}}\left({2\sum_{j = 1}^{N - k}{\sin \frac{{(j + 1)\pi}}{{N + 2}}\sin \frac{{(j + k)\pi}}{{N + 2}}}}\right. \\
\left.{- \left({N - k - 1}\right)\sin \frac{{k\pi}}{{N + 2}}\sin \frac{\pi}{{N + 2}}}\right).
\end{gather*}

Using the identity from   Lemma 6 we get
\begin{multline*}
\beta_{k - 1}^{(0)}- \beta_{k + 1}^{(0)}=\\
 \frac{{{U_{k - 1}}\left({{x_0}}\right)}}{{{{\sin}^2}{x_0}}}\frac{{\cos \frac{\pi}{{N + 2}}}}{{2\sin \frac{\pi}{{N + 2}}}}\left({\left({N - k + 3}\right)\sin \frac{{(k + 1)\pi}}{{N + 2}}
- \left({N - k + 1}\right)\sin \frac{{(k - 1)\pi}}{{N + 2}}}\right).
\end{multline*}
Further, using last formula from Lemma 3, and bearing in mind that 
\({U_{N - j}}\left({{x_0}}\right) ={U_j}\left({{x_0}}\right)\), we can write
\[
\beta_{k - 1}^{(0)}- \beta_{k + 1}^{(0)}= \cos \frac{\pi}{{N + 2}}{U_{k - 1}}\left({{x_0}}\right){U'_{N - k + 1}}\left({{x_0}}\right), 
\]
which yields
\begin{align*}
a_j^{(0)} &= \frac{{\beta_{j - 1}^{(0)}- \beta_{j + 1}^{(0)}}}{{\beta_0^{(0)}- \beta_2^{(0)}}}= \frac{1}{{N + 2}}\frac{{1 - x_0^2}}{{x_0^2}}\left({\cos \frac{\pi}{{N + 2}}{U_{k - 1}}\left({{x_0}}\right){{U'}_{N - k + 1}}\left({{x_0}}\right)}\right)  \\
 &= \frac{1}{{{{U'}_N}\left({\cos \frac{\pi}{{N + 2}}}\right)}}{{U'}_{N - j + 1}}\left({\cos \frac{\pi}{{N + 2}}}\right){U_{j - 1}}\left({\cos \frac{\pi}{{N + 2}}}\right).
\end{align*}
Here we have used \({U'_N}\left({\cos \frac{\pi}{{N + 2}}}\right) = \left({N + 2}\right)\frac{{\cos \frac{\pi}{{N + 2}}}}{{{{\sin}^2}\frac{\pi}{{N + 2}}}}\).

Thus, the optimal polynomial $F_N^{(0)}(z)$ has been constructed. Let us show that
the function \({\rho_N}(a_1, \ldots, a_N)\) 
attains its supremum  equal to 
\[
{\hat\rho_N^{(0)}=\max_{a_1=1}\left(\min_{t\in T\cup\pi}\Re(F_N^{(0)}(e^{it}))\right)}.
\]
To do that, consider the one-parameter family of conjugate trigonometric polynomials
\[
C^{(\varepsilon)}(t)=\frac{C^{(0)}(t)}{1+\varepsilon}+ \frac\varepsilon{1+\varepsilon}\cos t,\quad S^{(\varepsilon)}(t)=\frac{S^{(0)}(t)}{1+\varepsilon}+ \frac\varepsilon{1+\varepsilon}\sin t,
\]
which obviously satisfy $a_1=1$.

For all $t\in(0,\pi)$ and $\varepsilon>0$ it is obvious that $S^{(\varepsilon)}(t)>0.$ Since 
\[
C^{(\varepsilon)}(\pi)=\frac{\hat\rho_N^{(0)}}{1+\varepsilon}- \frac\varepsilon{1+\varepsilon},
\]
we have $C^{(\varepsilon)}(\pi)<\hat\rho_N^{(0)}$ and $C^{(\varepsilon)}(\pi)\to\hat\rho_N^{(0)}$ as $\varepsilon\to 0.$ Therefore
\begin{multline*}
\sup_{{a_1}= 1}\left({\min_t \left\{{\Re\left({F_N\left({{e^{it}}}\right)}\right):\Im \left({F_N\left({{e^{it}}}\right)}\right) = 0}\right\}}\right) \\
=\max_{{a_1}= 1}\left({\min_{t \in T \cup \left\{\pi \right\}}\left\{{\Re\left({F_N^{(0)}\left({{e^{it}}}\right)}\right)}\right\}}\right).
\end{multline*}
The theorem is proved.
\end{proof}

\begin{corollary}
\[
\inf_{a_2,...,a_N}\left({\max_t \left\{{\Re\left({F\left({{e^{it}}}\right)}\right):\Im \left({F\left({{e^{it}}}\right)}\right) = 0}\right\}}\right) = \frac{1}{4}{\sec ^2}\frac{\pi}{{N + 2}}.
\] 
The coefficients of the extremal polynomial are
\[
a_j^{(0)}= \frac{{{{( - 1)}^{j + 1}}}}{{{{U'}_N}\left({\cos \frac{\pi}{{N + 2}}}\right)}}{U'_{N - j + 1}}\left({\cos \frac{\pi}{{N + 2}}}\right){U_{j - 1}}\left({\cos \frac{\pi}{{N + 2}}}\right).
\]
\end{corollary}

\section{The optimization problem and coverings}

The first and most important result about covering  line segments and circles by a conformal mapping of the  unit disc  \(\mathbb D =\{{z:\left| z \right| < 1}\}\) is the well-known theorem of Koebe, whose complete proof was given by Bieberbach in 1916. According to that theorem the image of $\DD$ under a univalent mapping with standard normalization contains a central disc of radius 1/4. The sharpness of this constant is witnessed by the function $e^{-i\theta}K(ze^{i\theta}),$ where
\begin{equation}\label{13}
K(z)=\frac z{(1-z)^2}
\end{equation}
 is called the Koebe function.

Several generalizations of covering theorems were given in \cite{18.,19.,20.}. In \cite{9.} the Koebe theorem was generalized to nonunivalent functions with real coefficients by showing that the minimal simply connected set that contains the image of $\DD$ contains a central disc of radius 1/4. The proof follows from the Riemann conformal mapping theorem and the Lindel\"of principle: the minimal simply connected set that contains the image of $\DD$ can be conformally mapped onto $\DD$ by a function whose derivative at the origin is greater than or equal to~1  \cite{21.}.

Some such covering  theorems can be proved for polynomials \cite{22.}. For univalent polynomials \(F(z)=z+\sum_{j=2}^Na_jz^j\)  with real coefficients 
the estimate of the Koebe radius $R$ can be obtained via the Rogosinski-Szeg\"o estimate  $|a_2|\le2\cos2\psi_N,$ where $\psi_N=\pi/(N+3)$ if $N$ is odd, and $\psi_N$ is the smallest positive root of the equation
$$
(N+4)\sin(N+2)\psi_N+(N+2)\sin(N+4)\psi_N=0
$$ 
if $N$ is even \cite{23.}. One can check that in the even case, 
$\pi/(N+3)<\psi_N<\pi/(N+2).$ Hence the 
Koebe radius has the estimate
\begin{equation}\label{14}
	R\ge \frac1{2+\max\{|a_2|\}}\ge\frac1{4\cos^2\frac\pi{N+3}}.
\end{equation}
This estimate from below  can be completed by an estimate from above by using the results of Section 3.

\begin{theorem}The minimum simply connected set that contains the image of $\DD$ under the polynomial mapping \(F\left( z \right) = z + \sum_{j = 2}^N{{a_j}{z^j}}\) contains the interval
\begin{equation}\label{15}
\left({- \frac{1}{4}{{\sec}^2}\frac{\pi}{{N + 2}},\; \frac{1}{4}{{\sec}^2}\frac{\pi}{{N + 2}}}\right). 
\end{equation}
If the polynomial \(F\left( z \right) = z + \sum_{j = 2}^N{{a_j}{z^j}}\) is typically real, then the interval \eqref{15} is covered by \(F\left(\mathbb D \right)\).
\end{theorem}

The interval 
\eqref{15} is best possible, as shown by  
\[
F\left( z \right) = z + \sum_{j = 2}^N{{{( - 1)}^{j - 1}}{a_j}{z^j}}, 
\]
where
\[
a_j^{(0)}= \frac{1}{{{{U'}_N}\left({\cos \frac{\pi}{{N + 2}}}\right)}}{U'_{N - j + 1}}\left({\cos \frac{\pi}{{N + 2}}}\right){U_{j - 1}}\left({\cos \frac{\pi}{{N + 2}}}\right).
\]

\begin{corollary} 
The Koebe radius $R$ for typically real polynomials satisfies
\begin{equation}\label{16}
R \le \frac{1}{4\cos^2\frac\pi{N + 2}}. 
\end{equation}
\end{corollary}

Note that we cannot apply the Riemann Theorem and the Lindel\"of principle here because there is no guarantee that the univalent function that maps the minimal simply connected set containing the image of $\DD$ onto $\DD$ is a polynomial of degree 
$\le N$. 

In \cite{11.,12.} the Suffridge polynomial was used to approximate the Koebe function \eqref{13}. Several extremal properties of  the Suffridge polynomials were established in \cite{13.,14.}. 
In particular, the following is valid.

\begin{theorem}
The minimal simply connected set that contains the image of $\DD$ under the polynomial map \(F\left( z \right) = \sum_{j = 1}^N{{a_j}{z^j}}\) with \(F(1)=1\) contains the interval
\begin{equation}\label{17}
\left(-\tan^2\frac\pi{2(N+2)},1\right).
\end{equation}
If  $F(z)$ is typically real then the interval \eqref{17}  is covered by 
$F(\DD)$.
\end{theorem}

In general, we cannot replace the interval \eqref{17} by a disc in this result. Indeed, for even $N$ the image of $\DD$ under the univalent polynomial map \(F\left( z \right) = \sum_{j = 1}^N{{a_j}{z^j}}\) with \( F(1)=1\) does contain the central disc of radius \(\tan^2\frac\pi{2(N+2)}.\) For odd $N$, on the contrary, 
the radius equals $|F(e^{i\tau_N})|,$ where $\tau_N$ is the root of the equation $F^\prime(e^{it})=0$ that is closest to $\pi.$

\section{Connection with nonnegative trigonometric polynomials}
Nonnegative trigonometric polynomials appear in many problems of real analysis, theory of univalent maps, approximation theory, theory of orthogonal polynomials on the unit cricle, number theory and other branches of mathematics. The most natural and nicest examples of application of nonnegative trigonometric polynomials occur in solving extremal problems.

Nonnegative trigonometric polynomials $1+\sum_{j=1}^Na_j\cos jt$ satisfy the Fej\'er inequality \cite{15.,24.}
$$
|a_1|\le2\cos\frac\pi{N+2}.
$$
The inequality is sharp, with the unique extremal polynomial 
$$
\frac2{N+2}\sin^2\frac\pi{N+2}\Phi_N^{(1)}(t),
$$
where
$$
\Phi_N^{(1)}(t)=\left(\frac{\cos\frac{N+2}2t}{\cos t-\cos \frac\pi{N+2}}\right)^2.
$$
The above polynomial has coefficients \cite{25.}
\begin{multline*}
a_j^{(0)}=\frac1{N+2}\csc\frac\pi{N+2}\\
\cdot\bigg(
(N-j+3)\sin\frac{\pi(j+1)}{N+2}-(N-j+1)\sin\frac{\pi(j-1)}{N+2}
\bigg),
\end{multline*}
or by Lemma 3,
$$
a_j^{(0)}=\frac{2\sin^2\frac\pi{N+2}}{N+2}U^\prime_{N-j+1}\left(\cos\frac\pi{N+2}\right).
$$
In 1900  Fej\'er \cite{26.} established that the trigonometric polynomial  
$$
1+2\sum_{j=1}^N\left(1-\frac j{N+1}\right)\cos jt
$$
is nonnegative by proving for it the representation
$$
1+2\sum_{j=1}^N\left(1-\frac j{N+1}\right)\cos jt=\frac1{N+1}\Phi_N^{(2)}(t),
$$
where
$$
\Phi_N^{(2)}(t)=\left(\frac{\sin\frac{N+1}2t}{\sin \frac t2}\right)^2.
$$
Further, Fej\'er \cite{15.,27.} established an extremal property of  $\Phi_N^{(2)}(t)$: the maximum of a nonnegative trigonometric polynomial $1+\sum_{j=1}^N a_j\cos jt$ does not exceed $N+1$, with equality attained solely for $\frac1{N+1}\Phi_N^{(2)}(t)$ and only at the points $2\pi k$, $k=0,\pm1,\pm2,\ldots.$

The functions $\Phi_N^{(1)}(t)$ and $\Phi_N^{(2)}(t)$ are called Fej\'er kernels and have various extremal properties \cite{13.,14.}.

Let us consider the following problem: for a trigonometric polynomial
$C(t)=\sum_{j=1}^N a_j\cos jt$ with real coefficients and $a_1=1$ find
$$
\max_{a_j}\min_t C(t) .
$$

\begin{theorem}

\begin{equation}\label{18}
\max_{a_j}\min_t C(t) =-\frac12\sec\frac\pi{N+2}.
\end{equation}
The solution is unique and is given by 
$$
a_j^{(0)}=\frac{U_{N-j+1}^\prime\left( \cos\frac\pi{N+2}\right)}{U_{N}^\prime\left( \cos\frac\pi{N+2}\right)},\quad j=1,...,N.
$$
\end{theorem}

\begin{proof}
Suppose the trigonometric polynomial $\gamma+\cos t+\sum_{j=2}^Na_j\cos jt$ is nonnegative. Then 
$$
1+\frac{\cos t}\gamma+\sum_{j=2}^N\frac{a_j}\gamma\cos jt\ge0.
$$ 
Thus, from the Fej\'er inequality we get $\frac1\gamma\le2\cos\frac\pi{N+2}$, or
$\gamma\ge\frac12\sec\frac\pi{N+2}$, with equality attained for
$$
\frac{a_j^{(0)}}\gamma=\frac{2\sin^2\frac\pi{N+2}}{N+2}U^\prime_{N-j+1}\left(\cos\frac\pi{N+2}\right).
$$
Hence
\[
a_j^{(0)}=\frac12\sec\frac\pi{N+2}\frac{2\sin^2\frac\pi{N+2}}{N+2}U^\prime_{N-j+1}\left(\cos\frac\pi{N+2}\right)=\frac{U_{N-j+1}^\prime\left( \cos\frac\pi{N+2}\right)}{U_{N}^\prime\left( \cos\frac\pi{N+2}\right)}.\qedhere
\]
\end{proof}

Note that the Fej\'er trigonometric polynomial (with appropriate normalization) 
$$
\frac{N+1}{2N}\left(-1+\frac1{N+1}\Phi_N^{(2)}(t)\right)=\sum_{j=1}^N\left(1-\frac{j-1}N\right)\cos jt
$$
gives the estimate $\frac{N+1}{2N}$, which is worse than \eqref{18}.

\begin{figure}
\includegraphics[scale=0.25]{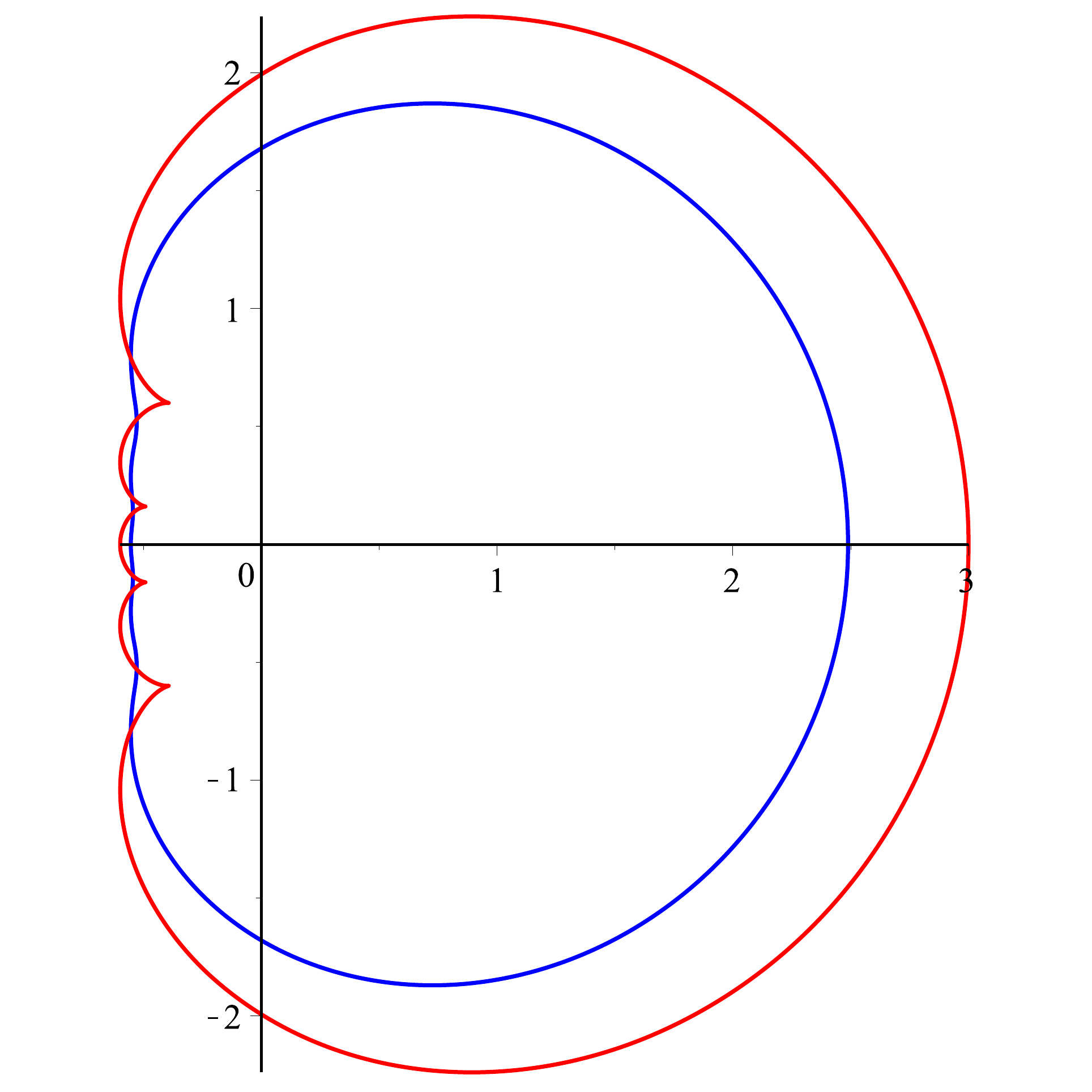}\hspace{.4cm}
\includegraphics[scale=0.25]{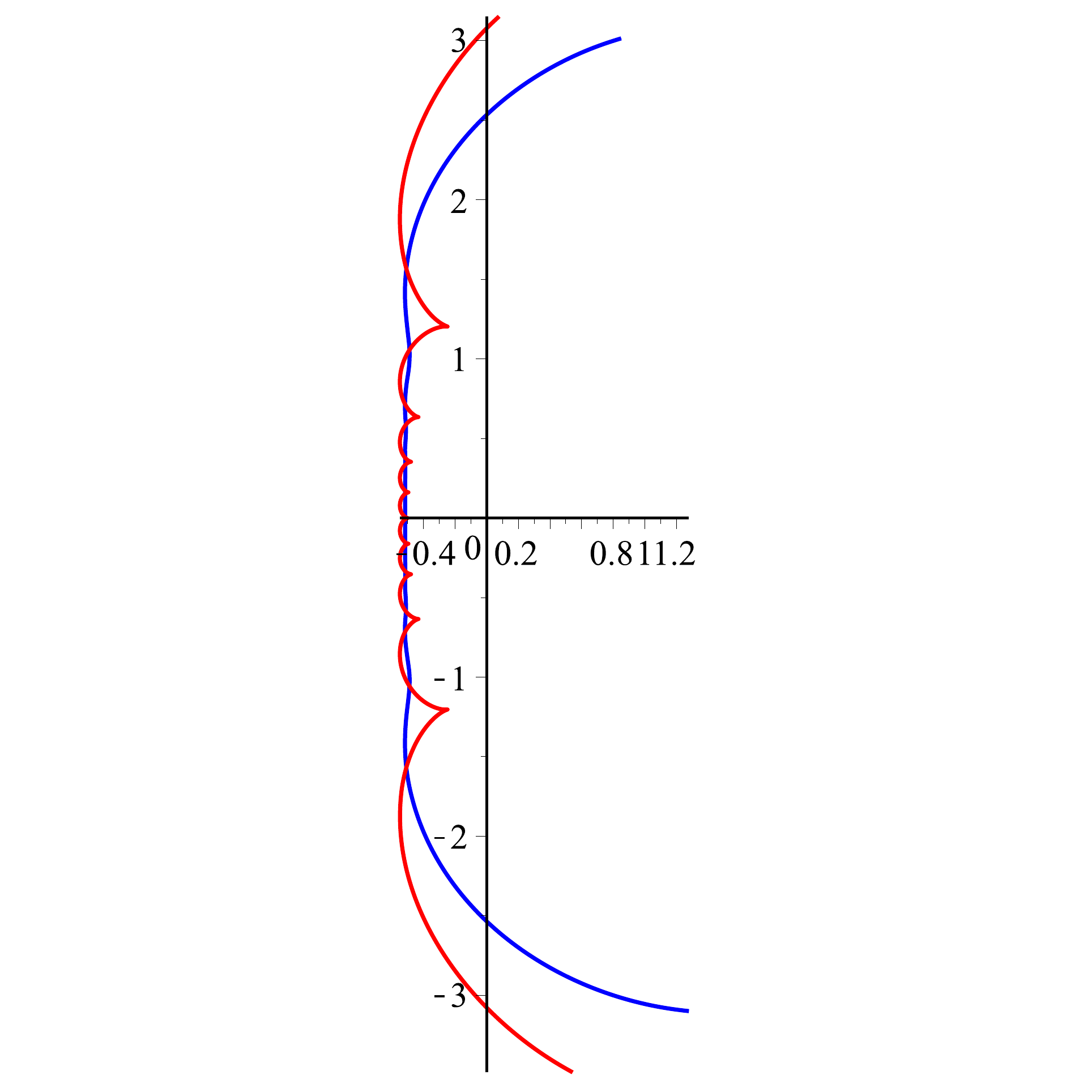}
\caption{The image of the unit disc under the Fej\'er polynomial $\hat\Phi(z)=\sum_{j=1}^N\left(1-\frac{j-1}N\right)z^j$ (red) and the polynomial $\frac1{U^\prime_N\left(\cos\frac\pi{N+2}\right)}\sum_{j=1}^NU^\prime_{N-j+1}(\cos\frac\pi{N+2})z^j$ (blue) for $N=5$ (left) and $N=10$ (right).}
\end{figure}

\begin{figure}
\includegraphics[scale=0.25]{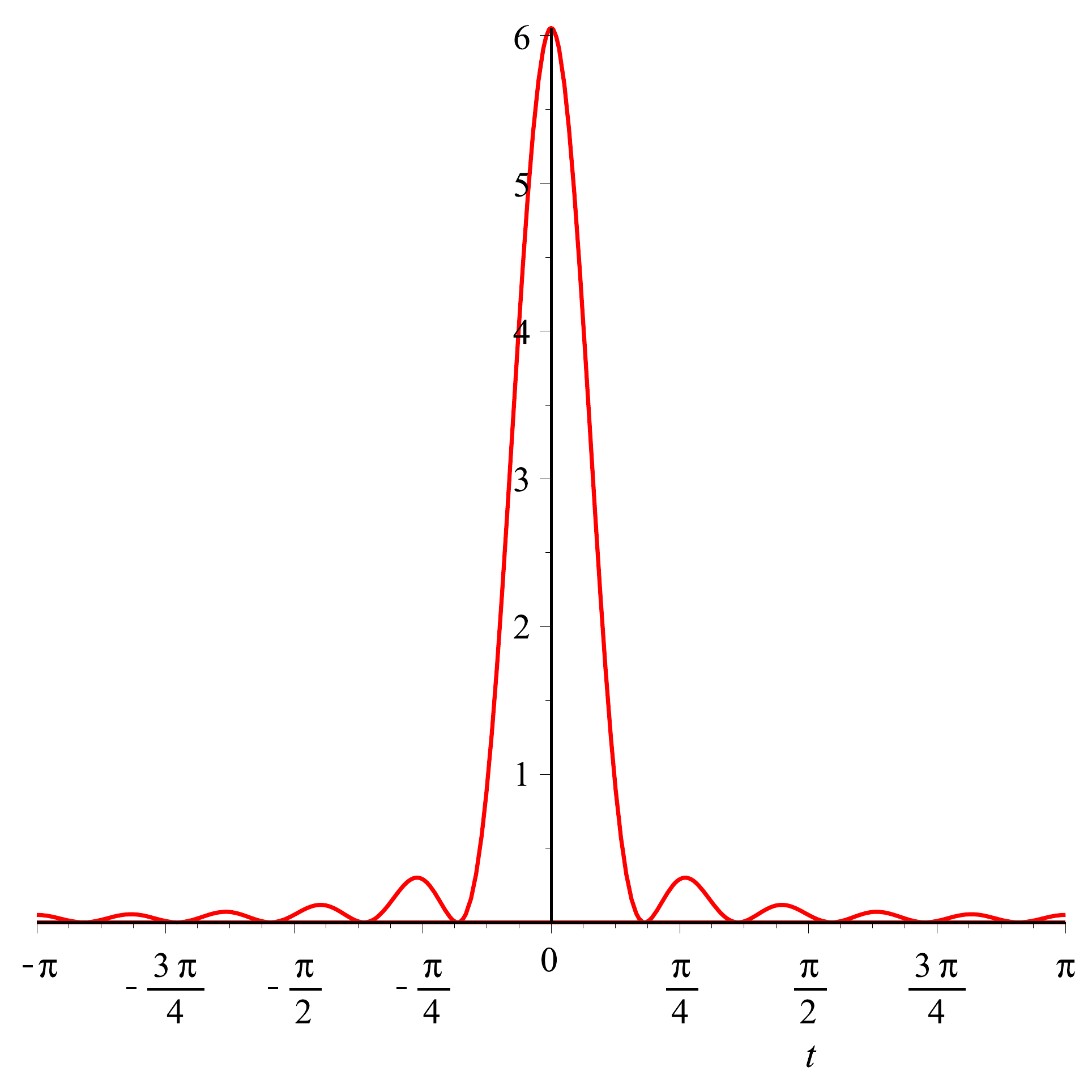}\hspace{.4cm}
\includegraphics[scale=0.25]{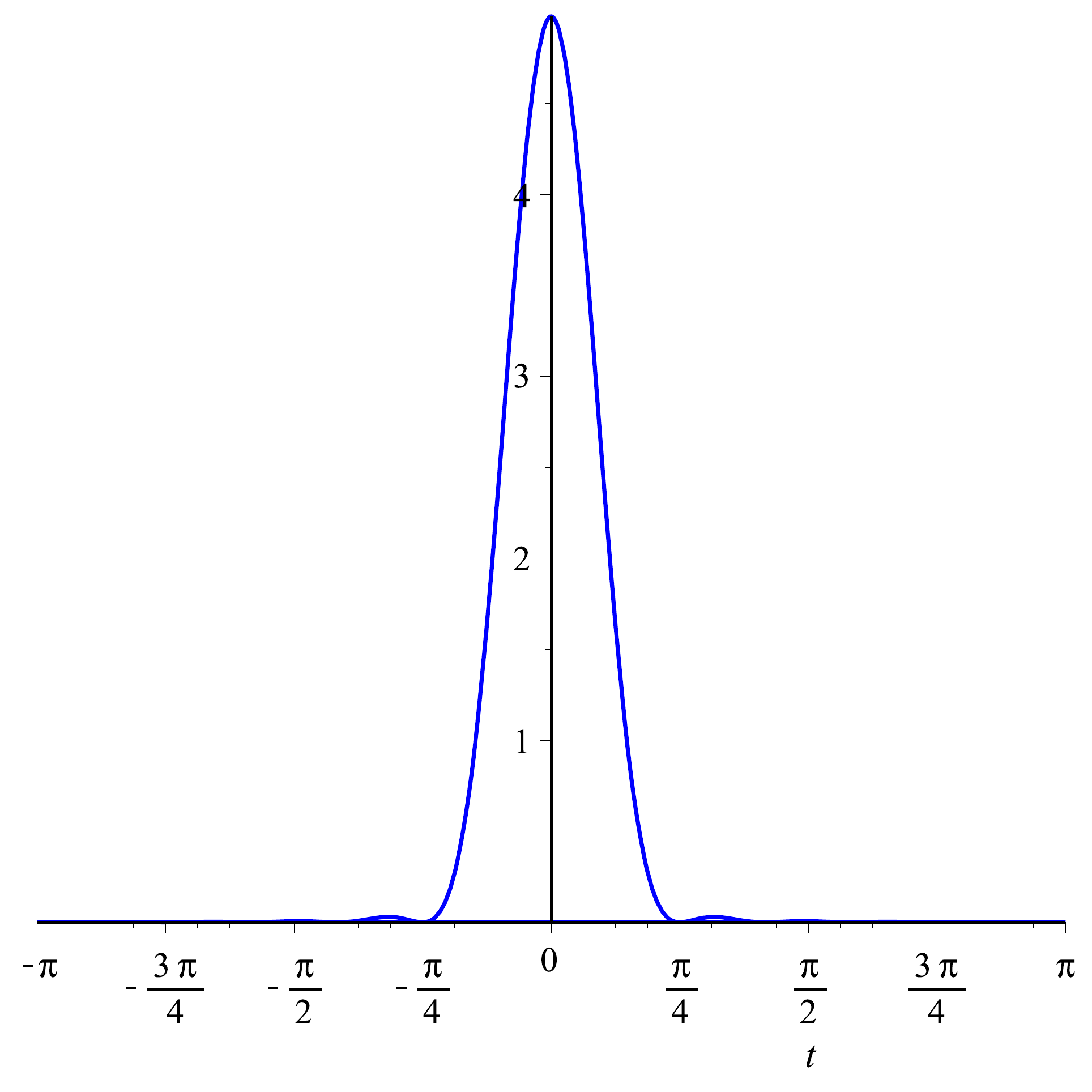}
\caption{The Fej\'er kernel $\sum_{j=1}^N\left(1-\frac{j-1}N\right)\cos jt$ (red) and the polynomial $\frac1{U^\prime_N\left(\cos\frac\pi{N+2}\right)}\sum_{j=1}^NU^\prime_{N-j+1}\left(\cos\frac\pi{N+2}\right)\cos jt$ (blue) for $N=10.$}
\end{figure}

Let us note that a different normalization leads to another extremal polynomial:
$$
\max_{\sum_{j=1}^Na_j=1}\min_t C(t) =-\frac{1}{N}, \quad a_j^{(0)} = \frac{2(N+1-j)}{N(N+1)}, \quad j = 1, \ldots, N,
$$
which is the (suitably normalized) classical Fej\'er polynomial 
$$
\frac1N\left(-1+\frac1{N+1}\Phi_N^{(2)}(t)\right).
$$

\section{Final Conjectures}

If \(N = 3, 4\) then the polynomials
\[
F\left( z \right) = \frac{1}{{{{U'}_N}\left({\cos \frac{\pi}{{N + 2}}}\right)}}\sum_{j = 1}^N{{{U'}_{N - j + 1}}\left({\cos \frac{\pi}{{N + 2}}}\right){U_{j - 1}}\left({\cos \frac{\pi}{{N + 2}}}\right){z^j}}
\]
take the form
$$
F_3(z)=z+\frac2{\sqrt5}z^2+\frac12\left(1- \frac1{\sqrt5}\right)z^3,\quad F_4(z)=z+\frac76z^2+\frac23z^3+\frac16z^4.
$$
The polynomial $F_3(z)$ is univalent in $\DD$ because the point $\left(\frac2{\sqrt5},\frac12\left(1- \frac1{\sqrt5}\right)\right)$ is in the  univalence region  of  cubic polynomials in the space of coefficients $(a_2,a_3)$ \cite{28.}.
For $N=3$ the Koebe radius is 
$$
R=\frac14\sec^2\frac\pi5.
$$
A nice proof of the univalence of  $F_4(z)$ is found in \cite{29.}. Numerical simulations suggest univalence for all \(N\). This justifies 

\begin{conjecture}\label{con1}  The polynomial
\begin{equation}\label{19}
F\left( z \right) = \frac{1}{{{{U'}_N}\left({\cos \frac{\pi}{{N + 2}}}\right)}}\sum_{j = 1}^N{{{U'}_{N - j + 1}}\left({\cos \frac{\pi}{{N + 2}}}\right){U_{j - 1}}\left({\cos \frac{\pi}{{N + 2}}}\right){z^j}}
\end{equation}
is univalent in  \(\mathbb D\).
\end{conjecture}

\begin{figure}
\includegraphics[scale=0.25]{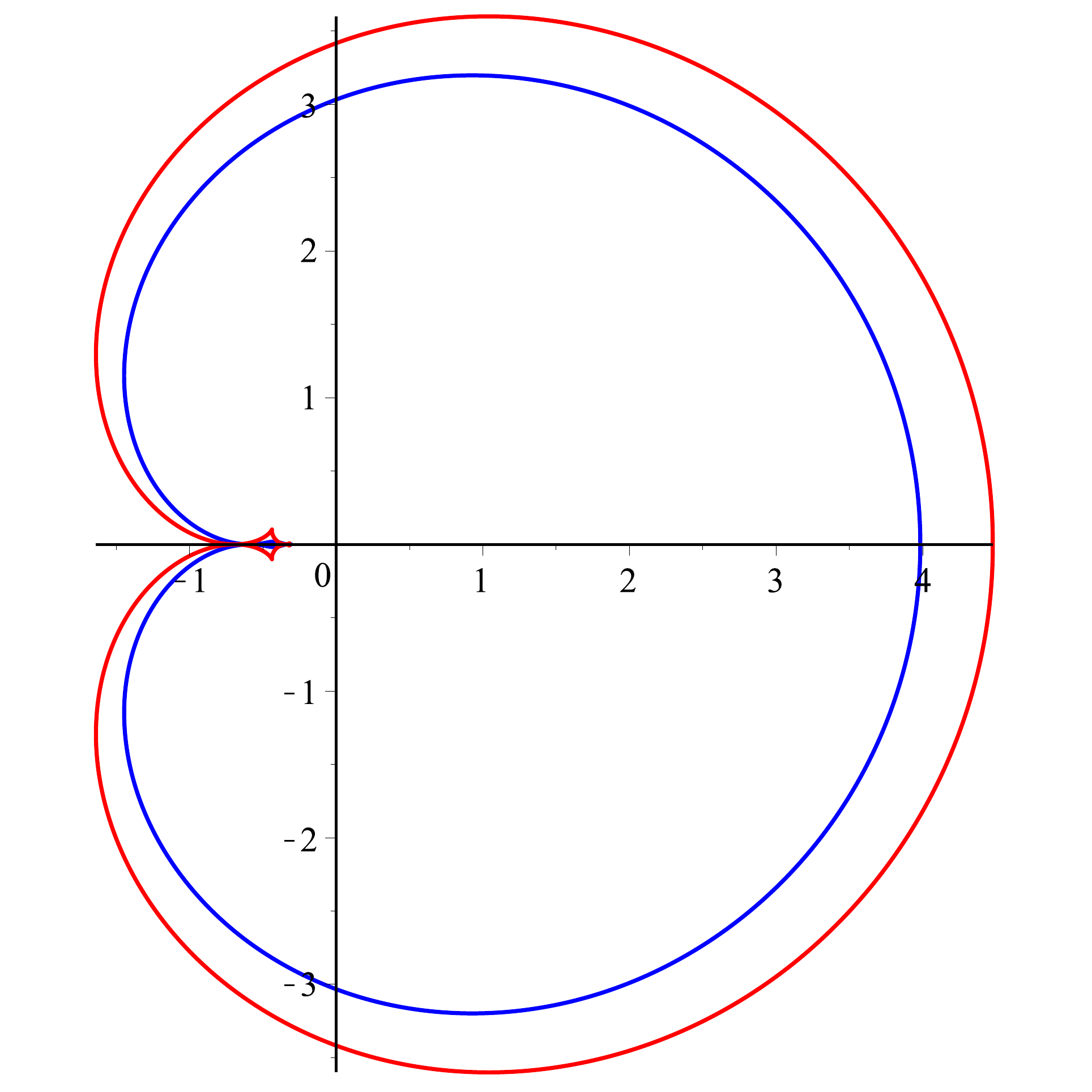}\hspace{.4cm}
\includegraphics[scale=0.25]{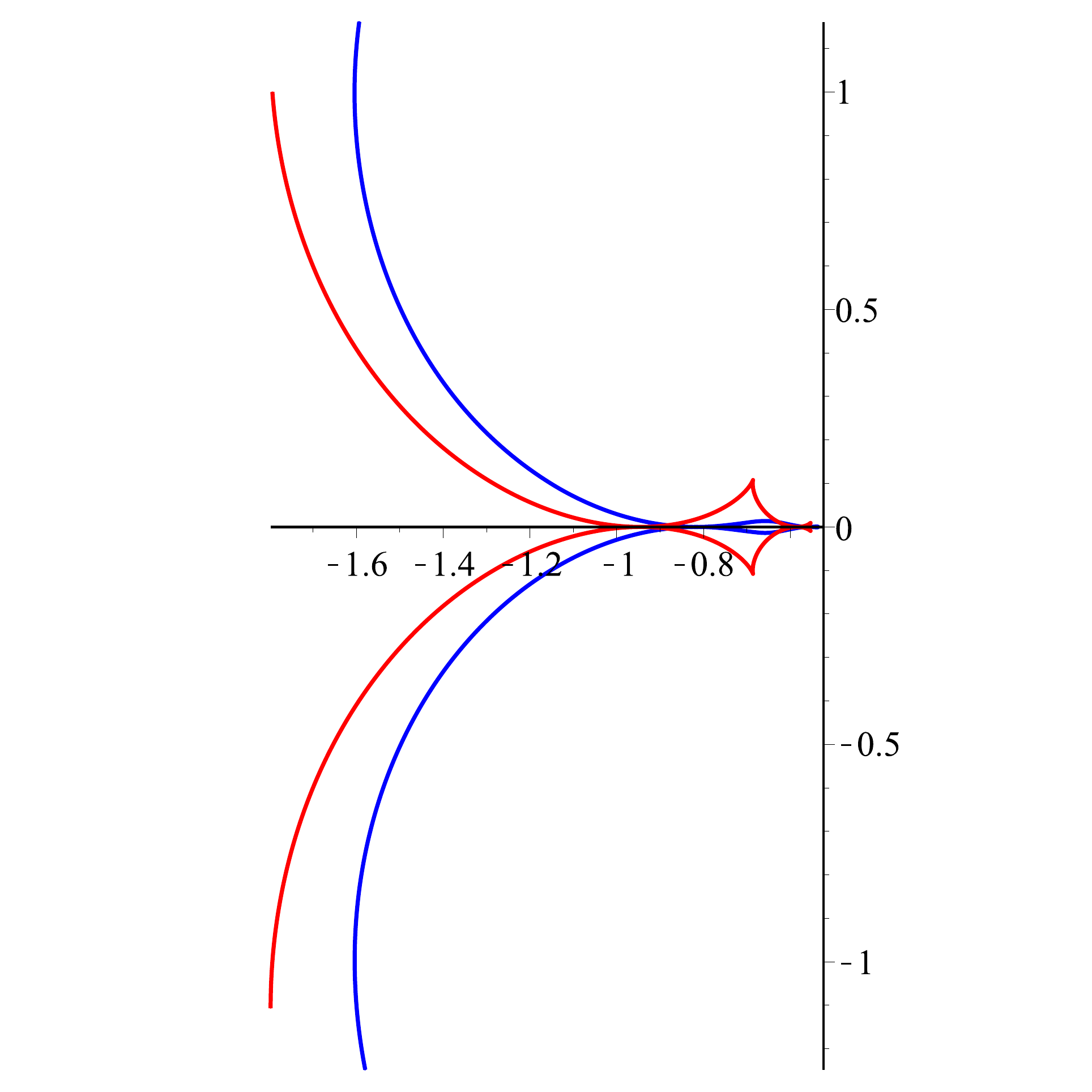}
\caption{The images if the unit disc under the Suffridge (red) polynomial and under \eqref{19} (blue) for $N=5$.}
\end{figure}

\begin{conjecture}\label{con2} 
The  Koebe radius of every univalent polynomial of degree \(N\) is 
\[
R = \frac{1}{4\cos^2(\pi/(N+2))}.
\]
\end{conjecture}

\begin{conjecture}\label{con3} 
 Conjecture \ref{con2} remains valid for polynomials with complex coefficients. All optimal polynomials are given by  \({e^{-i\alpha}}F\left({z{e^{i\alpha}}}\right)\), where  \(F\left( z \right)\) is the polynomial \eqref{19}.
\end{conjecture}

In \cite{11.}, together with  \eqref{2} the author considers the polynomials
$$
S\left (q,z \right) = a_{1}^0\sum_{j = 1}^N{\left ({1 - \frac{{j - 1}}{N}}\right){U_{jq - 1}}\left ({\cos \frac{\pi}{{N + 1}}}\right){z^j}}
$$
where $q=1,\ldots,N.$ It is shown that each such  polynomial is univalent.

\begin{conjecture}\label{con4}  The polynomial
\begin{align*}
F(q,z) &=\frac{1}
{U_{q-1}\left (\cos \frac\pi{N + 2}\right) U^\prime_N \left({\cos \frac{\pi}{{N + 2}}}\right)}\\
&\quad\times\sum_{j=1}^N {U'_{N - j + 1}} \left ({\cos \frac{\pi}{{N + 2}}}\right){U_{jq - 1}}\left ({\cos \frac{\pi}{{N + 2}}}\right)z^j
\end{align*}
is univalent for every $q = 1, \ldots, N$ and $ N=1,2,3,\ldots.$.
\end{conjecture}

\begin{conjecture}\label{con5}  For a couple of conjugate polynomials
$$
C(t)=\sum_{j=1}^Na_j\cos\,(2j-1)t,\quad
S(t)=\sum_{j=1}^Na_j\sin\,(2j-1)t
$$
with real coefficients and the normalization condition $a_1=1$, we have
\begin{equation}\label{20}
\inf_{a_j}\max_t\,\{|C(t)|: S(t)=0\}=-\frac12\sec^2\frac\pi{2N+2}.
\end{equation}
The solution is unique and is given by
$$
a_j^{(0)}=\frac{U_{2(N-j+1)}^\prime\left( \cos\frac\pi{2N+2}\right)}{U_{2N}^\prime\left( \cos\frac\pi{2N+2}\right)},\quad j=1,\ldots,N.
$$
\end{conjecture}

Note that the Suffridge polynomial 
\[
S(N,z) = a_1^{(0)} \sum_{j = 1}^{2N-1} \left(1 - \frac{j-1}{2N-1}\right) U_{jN-1} \left(\cos\frac{\pi}{2N}\right) z^j
\]
produces the coefficients $a_j^{(0)}=\frac{2(N-j)+1}{2N-1}$, $j=1,\ldots,N$, and the estimate $\frac N{2N-1}$, which is worse than \eqref{20}.\\

It is interesting to mention that the Koebe radius for odd univalent functions (equal to 1/2) is twice that for all univalent functions. Apparently, the same result holds for polynomials.

\begin{figure}
\includegraphics[scale=0.25]{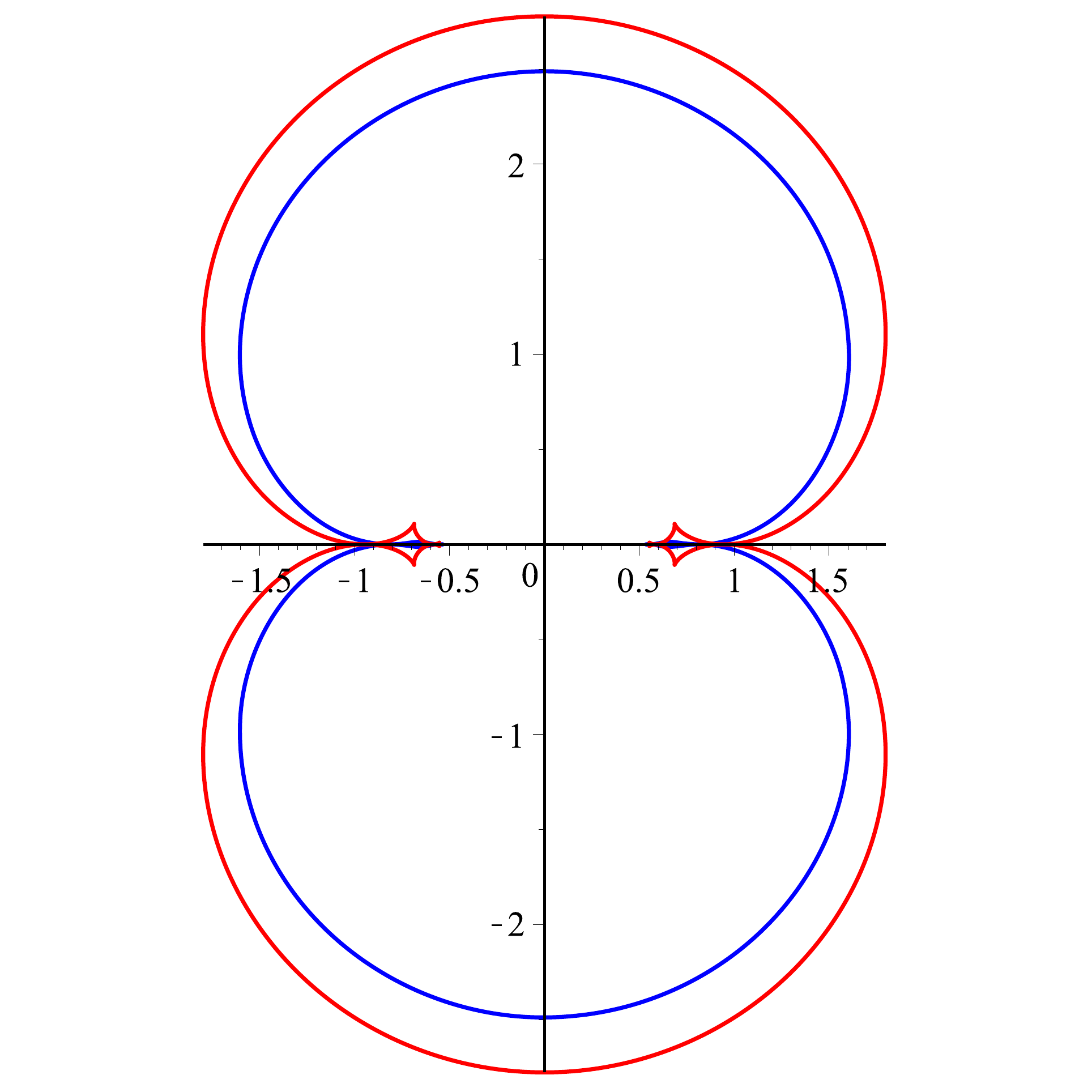}\hspace{.4cm}
\includegraphics[scale=0.25]{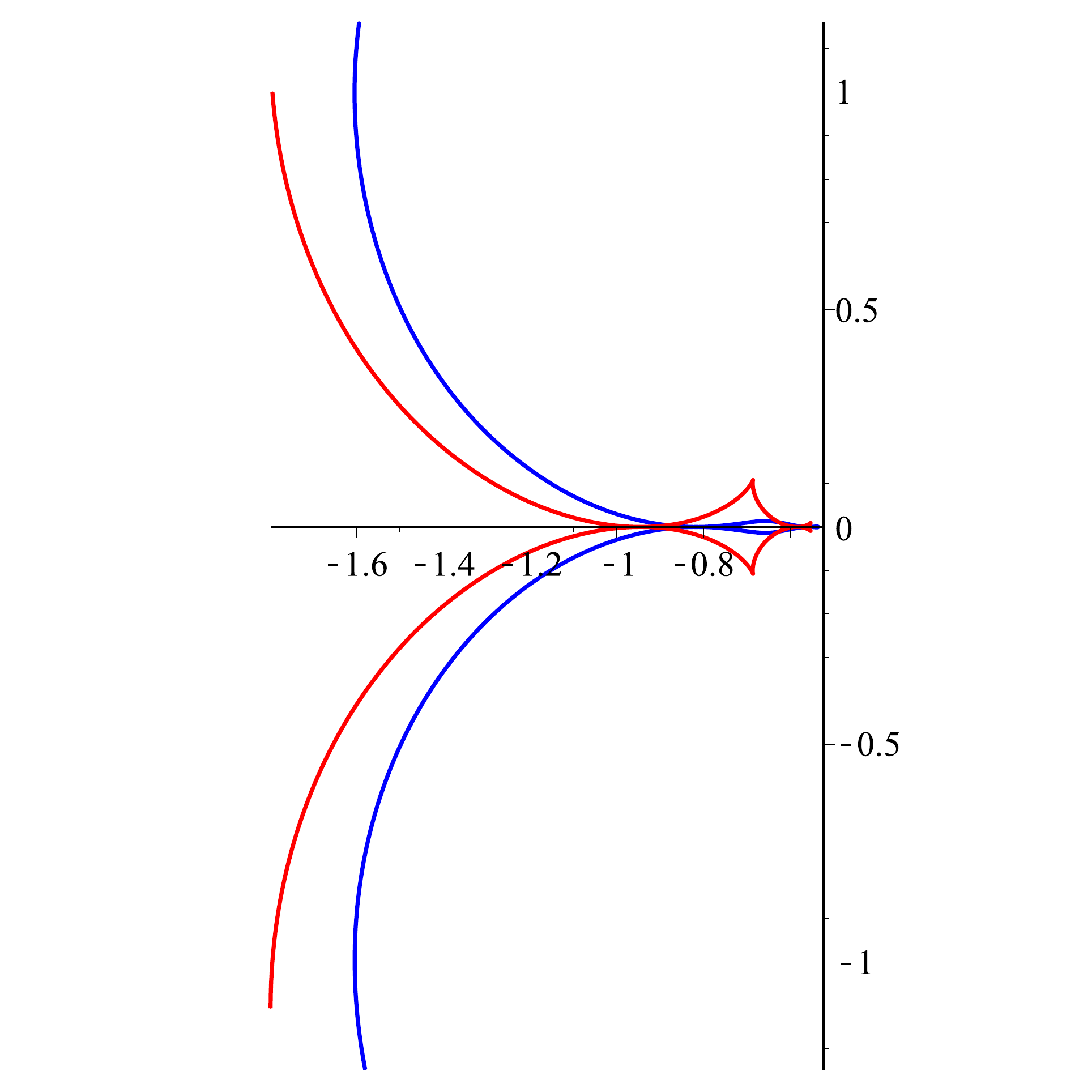}
\caption{
The image of $\DD$ under the Suffridge polynomial  $S(N,z)$ of  degree $2N-1$ (red) and the polynomial $F(N+1,z)$ of degree $2N$ (blue) for $N=5$.}
\end{figure}

\begin{conjecture}[Bounds for coefficients of univalent polynomials]\label{con6}  
 
For every positive integer \(k\), there exists \(N\) such that if \(F\left( z \right) = z + \sum_{j = 2}^M{{a_j}{z^j}}\) is univalent and $M\ge N$ then
\begin{equation}\label{con6}
\left|{{a_j}}\right| \le \frac{1}{{{U'_N}\left({\cos \frac{\pi}{{N + 2}}}\right)}}{U'_{N - j + 1}}\left({\cos \frac{\pi}{{N + 2}}}\right){U_{j - 1}}\left({\cos \frac{\pi}{{N + 2}}}\right)
\end{equation}
for all $j = 1, \ldots, k$.
\end{conjecture}

Finally, let us note that a right hand of the estimate \eqref{con6} monotonically approaches $j$ as $N\to\infty,$ which is an additional argument in favor of Conjecture 1.  

\section{Acknowledgment} The authors would like to thank Konstantin Dyakonov for fruitful discussions and Jerzy Trzeciak
for the help in preparation of the manuscript.


\begin{thebibliography}{99}


\bibitem{2.} G. Chen and X. Dong {\it From chaos to order: Methodologies, Perspectives and Application}. World Scientific, Singapore (1999).

\bibitem{12.}
A. Cordova and St. Ruscheweyh, {\it On the maximal range problem for slit domains}, Proceedings of a Conference, Valpara{\i}so, 1989, Lecture Notes in Mathematics, Vol. 1435, Springer, Berlin, 1990, pp. 33-44.

\bibitem{28.}
V.F. Cowling and W.C. Royster, {\it Domains of variability for univalent polynomials}, Proc. Amer. Math. Soc., 19 (1968), pp. 767-772.

\bibitem{4.} d.S.M. Vieira and A.J. Lichtenberg  {\it Controlling chaos using nonlinear feedback with delay}. Phys. Rev. E 54, pp.  1200-1207 (1996).

\bibitem{29.} J. Dillies, {\it Univalence of a certain quartic function}. arXiv:1803.03098.

\bibitem{13.}
D. Dimitrov, {\it Extremal Positive Trigonometric Polynomials}. Approximation Theory: A Volume Dedicated to Blagovest Sendov (B. Bojoanov, ed.): DARBA, 2002, pp.  1-24.

\bibitem{5.} D. Dmitrishin, A. Stokolos, I. Skrynnik, and E. Franzheva,
{\it Generalization of nonlinear control for nonlinear discrete systems.} arXiv:1709.10410.

\bibitem{8.}
 D. Dmitrishin, P. Hagelstein, A. Khamitova, and A. Stokolos. {\it Limitations of Robust Stability of a Linear Delayed Feedback Control}, SIAM J. Control Optim. 56,  2018,  pp. 148-157.

\bibitem{9.}
D. Dmitrishin, A. Khamitova, A. Stokolos, and M. Tohaneanu, {\it 
Finding Cycles in Nonlinear Autonomous Discrete Dynamical Systems}. Harmonic Analysis, Partial Differential Equations, Banach Spaces, and Operator Theory (Volume 2), Springer AWM series, Volume 5, 2017, pp. 199-237.

\bibitem{10.}
D. Dmitrishin and A. Khamitova, {\it Methods of harmonic analysis in nonlinear dynamics, Comptes Rendus Mathematique}, V. 351, Issues 9-10, pp. 367 - 370 (2013).

\bibitem{14.}
D. Dmitrishin, A. Khamitova, A. Korenovskyi and A. Stokolos, {\it Fejer and Suffridge Polynomials in the Delayed  Feedback Control Theory}. arXiv:1408.0163.

\bibitem{22.} V. N. Dubinin, {\it Russian Mathematical Surveys}, V. 67, №4 (2012),\  pp. 599-684.

\bibitem{25.}
E.V. Egervary and O. Szasz, {\it  Einige Extremalprobleme in Bereiche der trigonometrischen Polynomen}, Math. Z. 27 (1928), pp. 641-652.

\bibitem{24.}
Fejer L. {\it Ueber trigonometrische polynome}, J. fuer die reine und angew. Math. 1915. Bd. 146. S. pp. 53-82. 

\bibitem{26.}
L. Fejer, {\it. Sur les fonctions born\'ees et int\'egrables}, C. R. Acad. Sci. Paris, 131 (1900), pp. 984-987.

\bibitem{27.} 
L. Fejer, {\it  Sur les polynomes trigonom\'etriques},  C. R. Acad. Sci. Paris, 157 (1913), pp. 571-574.

\bibitem{20.}
M. Gregorczyk and L. Koczan, {\it A survey of a selection of methods for determination of Koebe sets}, Annales Universitatis Mariae Curie-Skłodowska Lublin - Polonia, Vol. Lxxi, No. 2, 2017 Sectio A, pp. 63-67.

\bibitem{21.}
G.M. Goluzin, {\it Geometric theory of functions of a complex variable.} Translations of Mathematical Monographs, V. 26, Providence, R.I. : American Mathematical Society , 1969.

\bibitem{19.}
J. A.  Jenkins, {\it On values omitted by univalent functions}. Amer. J. Math., V. 75, № 2, 1953, pp. 406-408.

\bibitem{16.} D.E. Longsine,  {\it Simultaneous Rayleigh-quotient minimization methods for Ax=$ \lambda $ Bx,}  1980, pp. 195-234

\bibitem{17.} R.S. Martin and J.H.Wilkinson, {\it Reduction of the symmetric eigenproblem \textit{Ax=$ \lambda $ Bx}\  and related problems to standard form}, Numer. Math. 11:  pp. 99-110 (1968). 

\bibitem{1.} E. Ott , C. Grebodgi, and J.A. Yorke, {\it Controlling chaos}. Phys. Rev. Lett. 64,  pp. 1196-1199 (1990).

\bibitem{15.}
 G. Polya and G. Szego {\it Problems and theorems in analysis. II. Theory of functions, zeros, polynomials, determinants, number theory}, geometry. Springer-Verlag, Berlin (1998). 392 p.

\bibitem{6.} K. Pyragas, {\it Continuous control of chaos by self controlling feedback.} Phys. Rev. Lett. A 170,  pp. 421-428 (1992).

\bibitem{3.} J.E.S. Socolar, D.W. Sukow, and D.J. Gauthier,  {\it   Stabilizing unstable periodic orbits in fast dynamical systems}. Phys. Rev. E 50, pp. 3245 - 3248 (1994).

\bibitem{18.}
D.С. Spenser, {\it On distortion in analytic transformations}, J. Math, and Phys., 20, № 1 (1941), pp. 124-126.

\bibitem{11.}
 T. Suffridge, {\it On univalent polynomials.} J. London Math. Soc. 44 (1969), pp. 496-504.

\bibitem{7.}
T. Ushio, {\it Limitation of delayed feedback control in nonlinear discrete - time systems}, IEEE Trans. Circ.\ Syst., 43, pp. 815-816  (1996).
 
\bibitem{23.} W.W. Rogosinski and G. Szego, {\it Extremum problems for non-negative sine polynomials}, Acta Sci. Math. Szeged 12 (1950), pp. 112-124.

\end{thebibliography}
\end{document}